\documentclass[11pt]{amsart}
\usepackage{graphicx}  
\usepackage{amssymb}
\usepackage{amsmath}
\usepackage{mathrsfs}
\usepackage[colorlinks=true,citecolor=black,linkcolor=black,urlcolor=blue]{hyperref}

\usepackage{verbatim}
\usepackage{needspace}
\input xy
\xyoption{all}

\usepackage{tikz}
\usetikzlibrary{decorations.pathreplacing,decorations.pathmorphing}
\usepackage{tkz-graph}
\usetikzlibrary{arrows}
\pgfarrowsdeclare{mytip}{mytip}{%
  \setlength{\arrowsize}{1pt}  
  \addtolength{\arrowsize}{.5\pgflinewidth}  
  \pgfarrowsrightextend{0}  
  \pgfarrowsleftextend{-5\arrowsize}  
}{%
  \setlength{\arrowsize}{1pt}  
  \addtolength{\arrowsize}{.5\pgflinewidth}  
  \pgfpathmoveto{\pgfpoint{-5\arrowsize}{1.5\arrowsize}}  
  \pgfpathlineto{\pgfpointorigin}  
  \pgfpathlineto{\pgfpoint{-5\arrowsize}{-1.5\arrowsize}}  
  \pgfusepathqstroke  
}

\newtheorem{thm}{Theorem}

\newtheorem{cor}[thm]{Corollary}
\newtheorem{prop}[thm]{Proposition}
\newtheorem{conj}[thm]{Conjecture}

\newtheorem{Definition}[thm]{Definition}
\newenvironment{definition}
  {\begin{Definition}\rm}{\end{Definition}}

\newtheorem{Example}[thm]{Example}
\newenvironment{example}
  {\begin{Example}\rm}{\hfill$\Box$\end{Example}}

\newtheorem{Remark}[thm]{Remark}

\newcommand{\Z}{\mathbb{Z}}
\newcommand{\R}{\mathbb{R}}

\title{Orientations, semiorders, arrangements, and parking functions}
\author{Sam Hopkins}
\email{hopkins@reed.edu}
\author{David Perkinson}
\email{davidp@reed.edu}
\address{Reed College, Portland OR, 97202}
\begin{document}
\begin{abstract}
It is known that the Pak-Stanley labeling of the Shi hyperplane arrangement
provides a bijection between the regions of the arrangement and parking
functions. For any graph $G$, we define the $G$-semiorder arrangement and show
that the Pak-Stanley labeling of its regions produces all $G$-parking functions.
\end{abstract}

\maketitle

In his study of Kazhdan-Lusztig cells of the affine Weyl group of
type $A_{n-1}$, \cite{shi}, J.-Y. Shi introduced the arrangement of hyperplanes in $\R^n$
now known as the Shi arrangement:
\[
x_i-x_j=0,1\qquad 1\leq i<j\leq n.
\]
Among other things, he proved that the number of regions in the complement of
this set of hyperplanes is $(n+1)^{(n-1)}$, Cayley's formula for the number of
trees on $n+1$ labeled vertices.  The first bijective proof of this fact is due
to Pak and Stanley, \cite{pak-stanley}, who provide a method for labeling the
regions with parking functions of size $n$.  Given a graph $G$, Postnikov and
Shapiro, \cite{postnikov}, introduced the notion of a $G$-parking function.  In
the case where there exists a vertex $q$ connected by edges to every other
vertex of $G$, Duval, Klivans, and Martin, \cite{duval}, have defined the
$G$-Shi arrangement, and conjecture that when its regions are labeled by the
method of Pak and Stanley, the resulting labels are exactly the $G$-parking
functions with respect to $q$.  In this case, however, there may be duplicates
among the labels.  Letting $G$ be the complete graph on $n+1$ vertices
recaptures the original result of Pak and Stanley.  Our work was motivated by
this conjecture.

Figure~\ref{schema} serves as a guide to our paper.  The four corners of the
square in Figure~\ref{schema} are labeled by structures we associate to a
graph $G$, which we now describe.
\begin{figure}[ht] 
\begin{tikzpicture}[scale=0.68]

  \draw (0,0) node(pf){\parbox[h]{1.5in}{\centering$\mathcal{S}$:
  quasi-superstables\\[3pt][Def.~\ref{def:quasi-superstable}]}};
  \draw (10,0) node(order){\parbox[h]{1.2in}{\centering$\mathbb{I}$:
  $\mbox{$G$-semiorders}$\\[3pt][Sect.~\ref{semiorder}]}};
  \draw (0,10) node(region){\parbox[h]{0.8in}{\centering$\mathcal{R}$:
  regions\\[3pt][Def.~\ref{semiorder arrangement}]}};
  \draw (10,10) node(orient){\parbox[h]{1.5in}{\centering$\mathbb{O}$:
  $G$-semiorientations\\[3pt][Def.~\ref{semiorientation}]}};
\draw[double, ->] (pf) to node[auto] {chip-firing} (order); 
\path (pf) to node[auto,swap] {[Sect.~\ref{ss-alg}]} (order); 
\draw[double, ->] (region) to node[auto,swap]
{\parbox[h]{0.8in}{\centering Pak-Stanley\\[3pt][Sect.~\ref{labeling}]}} (pf); 
\draw[double, <->] (region) to node[auto] {Farkas' lemma} (orient); 
\path (region) to node[auto,swap] {[Thm.\ref{hyperplane bijection theorem}]} (orient); 
\draw[double, ->] (order) to node[auto,swap]
{\parbox[h]{1.4in}{\centering $\nu$-compatibility\\[3pt][Sect.~\ref{semiorders and
semiorientations}]}} (orient); 
\draw[double, ->] (orient) to node[auto,swap]
{\parbox[h]{0.7in}{\centering $\mathrm{indeg}-1$\\[3pt][Thm.~\ref{main}]}} (pf); 
\end{tikzpicture}
\caption{Schematic diagram of results.}\label{schema}
\end{figure}
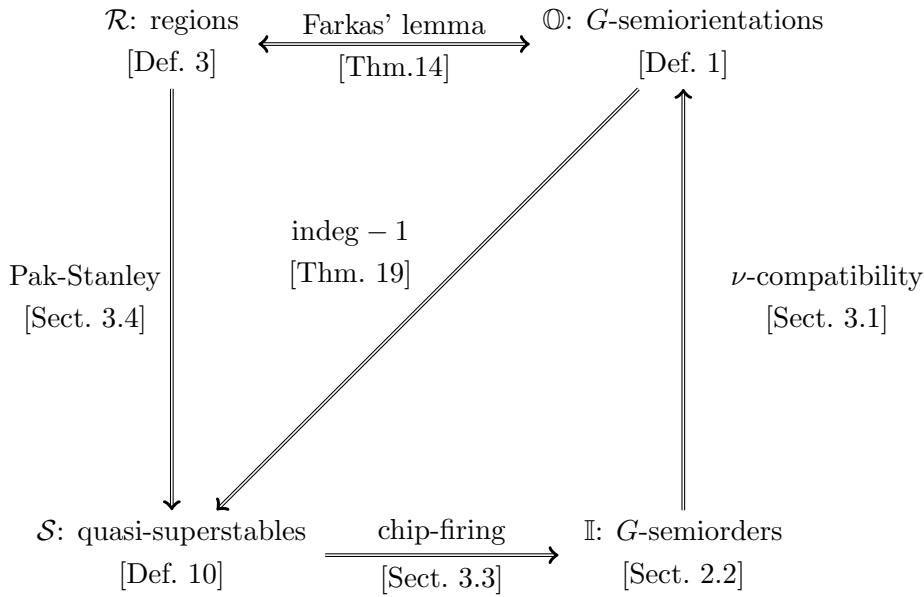
\medskip

\noindent{\bf Parking functions and superstable configurations.}  
Chip-firing is a key tool for us.  In the abelian sandpile model
for~$G$, one first chooses a vertex of $G$ to serve as the ``sink.''  Then placing
grains of sand (or chips) on each of the nonsink vertices defines a {\em
configuration}, $c$, on~$G$.  A vertex~$v$ is {\em unstable} in~$c$ if it has at
least as many grains of sand as its degree.  In that case, one may {\em fire}
(or {\em topple}) $v$ by sending one grain of sand from~$v$ to each of its
neighbors.  In this process neighbors of $v$ may become unstable, themselves.
Sand is not necessarily conserved under firing: grains of sand that are sent to
the sink vertex disappear.  For undirected, connected graphs---to which we limit
ourselves here---by repeatedly firing vertices, $c$ is eventually transformed
into a stable configuration, i.e., one with no unstable vertices.

Instead of firing one vertex at a time, one may consider a firing rule that
allows sets of vertices to be fired simultaneously.  This gives rise to a
stronger type of stability and a corresponding set of {\em superstable}
configurations (see Section~\ref{quasi-superstables}).  These superstable
configurations serve to define $G$-parking functions. A {\em $G$-parking
function} with respect to the chosen sink is a function from the vertices of $G$
to the integers whose value at the sink is $-1$ and whose values at the nonsink
vertices are the numbers of grains of sand for some superstable
configuration on $G$.  Thus, there is only a slight difference between
superstable configurations and $G$-parking functions, and at any rate, they are
in one-to-one correspondence.  Parking functions are more widely-known, and
they originally arose independently (for instance in \cite[p.~545]{knuth}). This 
accounts for why they, and not superstable configurations, are mentioned in the
title of this paper.  However, given the importance of chip-firing to our work,
we will mainly refer to superstable configurations outside of this introduction.

Although $G$-parking functions and superstable configurations are defined with
respect to a chosen sink vertex, we have found it natural and convenient to
first prove our results in a sinkless context, akin to working in projective
space rather than affine space.  To this end, we define generalized $G$-parking
functions, which we call {\em quasi-superstable divisors} on $G$.  These appear
in the bottom left corner of Figure~\ref{schema}, denoted by~$\mathcal{S}$.  As
indicated by the figure, quasi-superstables are formally defined in
Definition~\ref{def:quasi-superstable}.

\noindent{\bf Semiorders and the semiorder arrangement.} The {\em semiorder
arrangement}, \cite{stanley}, is the set of $n(n-1)$ hyperplanes in $\R^n$ given by
\[
x_i-x_j=1,\quad i,j\in\{1,\dots,n\}, \ i\neq j.
\]
Its regions are in bijection with certain $n$-element posets called {\em
semiorders}.  

A {\em $G$-semiorder} is a semiorder whose elements are the vertices of $G$.
The collection of $G$-semiorders, denoted by~$\mathbb{I}$, appears in the bottom
right corner of Figure~\ref{schema}.  In the same way that a Shi arrangement is
modified by Duval, Klivans, and Martin in~\cite{duval} to take into account the
structure of a graph, we modify a semiorder arrangement to produce the {\em
$G$-semiorder arrangement} associated with $G$.  The regions of the
$G$-semiorder arrangement are denoted by $\mathcal{R}$ in the top left corner of
Figure~\ref{schema}.

Fixing a sink vertex, $v$, we then refine the definition of a $G$-semiorder
arrangement to get the {\em $(G,v)$-semiorder arrangement}
(Definition~\ref{sink semiorder arrangement}).  
\medskip

\noindent{\bf Orientations.} A {\em partial orientation} of $G$ consists of
orienting some, not necessarily all, of the edges of $G$.  The {\em
$G$-semiorientations}, denoted by~$\mathbb{O}$ in the top right of
Figure~\ref{schema}, are the partial orientations satisfying an extra condition
relative to the cycles of the underlying graph, $G$. 
\medskip

\noindent{\bf Relations among the structures.} 
Using the method of Pak and Stanley, we label each region of the
$(G,v)$-semiorder arrangement with what amounts to a function from the vertices
of $G$ to the integers taking the value $-1$ at $v$.  The {\bf main goal of
this paper} is Theorem~\ref{thm:affine labeling}, which, along with
Theorem~\ref{thm:Pak-Stanley}, shows that the set of Pak-Stanley labels that are
negative only at the sink are exactly the set of $G$-parking functions.  The
set of Pak-Stanley labels consists of only $G$-parking functions exactly when
the sink is adjacent to every nonsink vertex.  

Our main goal, just described, is a consequence of the analogous result in the
nonsink context, represented by the left-most vertical arrow in
Figure~\ref{schema}.  This arrow represents the fact that the Pak-Stanley
labels for the regions of the $G$-semiorder arrangement form the set of
quasi-superstables, which is proved in Theorem~\ref{thm:labeling}.  This
theorem is proved by establishing the relations represented by the other arrows
in Figure~\ref{schema}.

In working on the $G$-Shi conjecture of Duval, Klivans, and Martin, we were led
to labeling the regions of the $G$-Shi arrangement with partial orientations of
$G$.  We used Farkas' lemma to develop a rule determining which partial orientations
would appear as labels.  However, it seemed that the criterion we developed was
awkward, and that it would become much simpler and symmetric if we altered the
hyperplane arrangement somewhat.  This is how we were led to consider
$G$-semiorder arrangements.  The resulting correspondence between the regions of
the $G$-semiorder arrangement and their labels with partial orientations (the
$G$-semiorientations) is represented by the top horizontal arrow in
Figure~\ref{schema} and is established in Theorem~\ref{hyperplane bijection
theorem}.

The correspondence between maximal superstable configurations---those containing
the most amount of sand---and acyclic orientations of the edges of $G$ has been
noted several times, in different forms (\cite{benson}, \cite{biggs},
\cite{goles}, \cite{greene}).  In~\cite{benson}, the correspondence is a
consequence of an extended version of Dhar's burning algorithm, a tool from the
chip-firing literature,~\cite{dhar2}.  The input to their extended algorithm is
a maximal superstable configuration, $c$.  The configuration is transformed into
an unstable configuration $\tilde{c}$.  The sequence of vertex firings that
stabilizes~$\tilde{c}$ is then used to orient the edges of $G$.  We modify this
algorithm in Section~\ref{ss-alg} to apply to all superstable configurations,
not just the maximal ones.  This algorithm is encoded (in its nonsink version)
as the bottom horizontal and right vertical arrows in Figure~\ref{schema} .

Ultimately, we label each region of the $G$-semiorder arrangement with (i) a
$G$-semiorientation, (ii) semiorders on the vertices of $G$, and (iii) a
quasi-superstable divisor on $G$.  Our aim---to show that all
quasi-superstables appear as labels---is achieved by showing that the upper-left
and the lower-right triangles in Figure~\ref{schema} commute and that the
mapping referred to as ``${\mathrm{indeg}-1}$'' along the diagonal in
Figure~\ref{schema} is surjective, (cf.~Theorem~\ref{thm:labeling},
Corollary~\ref{cor-to-main}, and Theorem~\ref{main}, respectively).  
\medskip

\noindent{\bf Organization.}  Following this introduction, the paper is
organized into four main sections and a conclusion.
Section~\ref{section:example} is an extended example that illustrates our main
results and may serve as a foundation for understanding them.  The rest of the
paper may be considered a justification of the claims made there.
Section~\ref{section:structures} defines our four main graphical structures.
Section~\ref{section:correspondences} contains the central results describing
the correspondences between these structures, as discussed above.
Section~\ref{sink} explains how to transfer the results of
Section~\ref{section:correspondences} to the context in which a sink vertex is
chosen.  The concluding section presents a conjecture and suggests further
lines of inquiry.  Imagine starting with the $G$-semiorder arrangement but then
perpendicularly displacing the hyperplanes (replacing the original hyperplanes
with parallel ones) before applying the labeling method of Pak and Stanley.  As
one slides the hyperplanes, some regions disappear and new regions form. We
conjecture that as long as a ``central region'' is preserved, the resulting set
of labels does not change.  A special case implies the $G$-Shi conjecture of
Duval, Klivans, and Martin.
\medskip

\noindent{\bf Acknowledgments.}  We thank Art Duval, Caroline Klivans, and
Jeremy Martin for encouraging us to work on the $G$-Shi conjecture and for
helpful comments.  We thank Collin Perkinson for help with proofreading.  We
also especially thank the anonymous referee, who helped to improve the quality
of our exposition.

\section{Introductory example}\label{section:example}
We introduce our main results with an example, beginning with
an explanation of the construction of Figure~\ref{fig:example}.  Consider the
arrangement of six planes in $\R^3$,
\[
x_i-x_j=1,\quad i,j\in\{1,2,3\},\ i\neq j.
\]
The complement of these planes in $\R^3$ consists of $19$ connected components.
The planes are all parallel to the vector $(1,1,1)$, so for our purposes
it suffices to intersect the arrangement with the perpendicular plane,
$x_1+x_2+x_3=0$, as pictured in Figure~\ref{fig:example}.  Our goal here is to
explain the labeling of the $19$ regions by vertex-labeled, partially oriented
graphs.

To start, each region contains a copy of the graph,~$G$,
pictured in Figure~\ref{fig:graph} after removal of the sink vertex, $q$.
\begin{figure}[ht] 
\begin{tikzpicture}
\SetVertexMath
\GraphInit[vstyle=Art]
\SetUpVertex[MinSize=3pt]
\SetVertexLabel
\tikzset{VertexStyle/.style = {%
shape = circle,
shading = ball,
ball color = black,
inner sep = 2pt
}}
\SetUpEdge[color=black]
\Vertex[LabelOut,Lpos=90,
Ldist=.1cm,x=0,y=1]{v_3}
\Vertex[LabelOut,Lpos=180,
Ldist=.1cm,x=-0.75,y=0]{v_1}
\Vertex[LabelOut,Lpos=0,
Ldist=.1cm,x=0.75,y=0]{v_2}
\Vertex[LabelOut,Lpos=270,
Ldist=.1cm,x=0,y=-1]{q}
\Edges(v_3,v_1,q,v_2,v_3)
\Edges(v_1,v_2)
\end{tikzpicture}
\caption{Graph $G$.}\label{fig:graph}
\end{figure}
The vertices of the remaining triangle are labeled by integers, and some of the
edges of the triangle are oriented.  We refer to the vertex labels as vectors,
$(c_1,c_2,c_3)$, where $c_i$ is the label for~$v_i$, and designate oriented
edges as ordered pairs of vertices, $(u,v)$, where $u$ is the tail and $v$ is
the head.  If there were more room, each region would be labeled by the full
graph, $G$.  In each region, the label on $q$ would be $-1$, and each edge
incident on $q$ would be oriented with tail at $q$.
\begin{figure} 
  \includegraphics[width=5in]{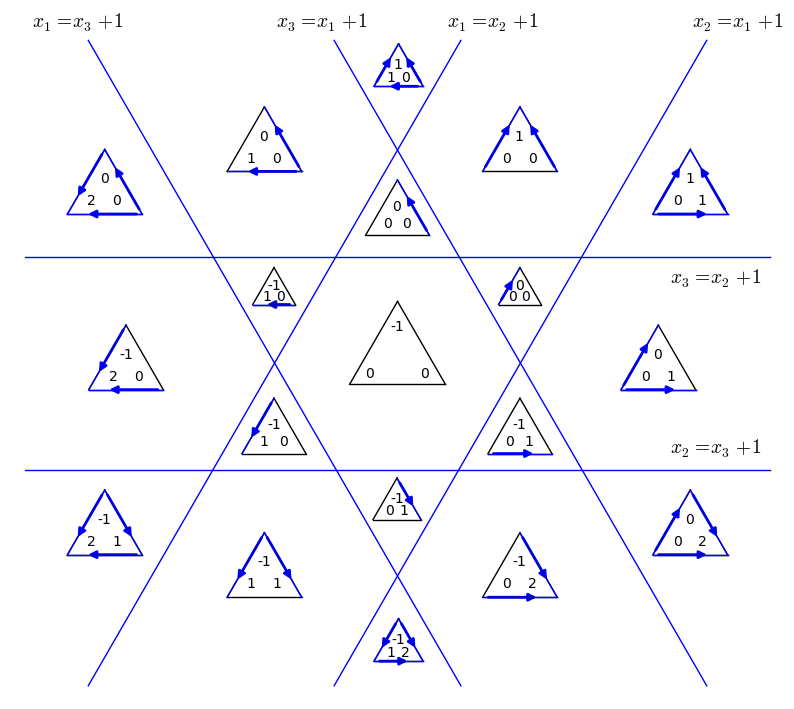}
\caption{Labeled $(G,q)$-semiorder arrangement.}\label{fig:example}
\end{figure}

The labeling of the regions is done inductively, starting at the center.  In
our example, the central region of the arrangement is a hexagon, and its
triangle has vertex labels $(0,0,-1)$ and no oriented edges (besides the
assumed ones from the sink).  The rule for the central region is that a
vertex~$v_i$ has a label of $0$ if $\{q,v_i\}$ is an edge; otherwise the label
is $-1$.  More generally, the rule for any region is that the vertex~$v_i$ is
labeled by one less then the number of oriented edges pointing into~$v_i$.

The central region shares an edge with six bordering regions.  Let $r$ denote
one of these regions, and say that $x_j-x_i=1$ is the border between it and the
central region.  Moving from the central region to $r$, one may think of the
value of $x_j$ as increasing at the expense of $x_i$.  We record this fact by
orienting the edge $\{v_i,v_j\}$ as $(v_i,v_j)$ and increasing the vertex label
for~$v_j$ by one.  Label the five remaining bordering regions similarly. For
example, moving into the compact region directly above the central region, we
cross into the region where $x_3>x_2+1$.  Thus, the label for this compact
region adds an oriented edge $(v_2,v_3)$, and the label for $v_3$ increases from
$-1$ to $0$.

At this point in the labeling process, there would be six unlabeled regions
bordering the six that were just labeled.  Let $r$ be one of these unlabeled
regions, and let $x_{\ell}-x_k=1$ be the edge it shares with a region $r'$ that
was just labeled.  In our case, there are two choices for $r'$, but this choice
does not affect the eventual label for $r$.  To label $r$, start with the
label for~$r'$, add the oriented edge $(x_k,x_{\ell})$, and increase the
vertex label for~$v_{\ell}$ by $1$.  After labeling these six regions, there
are six remaining unlabel regions, and these are labeled by continuing the
procedure just described.

In the end, each region is labeled with a copy of $G$ having labeled vertices
and a partial orientation of its edges.  The resulting partial orientations have
a special property.  Call an unoriented edge ``blank.'' Consider a collection of
edges,~$C$, forming a cycle in $G$.  Under any partial orientation, some of the
edges in $C$ will have an orientation.  If it is possible to orient the
remaining blank edges in $C$ to get a directed cycle, then the special property
is that there must be a greater number of blank edges in $C$ than oriented
edges: ``more blanks than arrows for potential directed cycles.''  For example,
consider the label for the region just above the central region.  It has one
oriented edge, $(v_2,v_3)$.  By orienting the two blank edges as $(v_1,v_2)$ and
$(v_3,v_1)$, we would get a directed cycle, but this potential cycle has two
blanks and only one arrow.  Theorem~\ref{hyperplane bijection theorem}
guarantees that the $19$ regions are in bijection with partial orientations
of~$G$ having this special property. 

As mentioned above, the vertex labels are given as one less than the indegree at
each vertex.  In our example, there are eight distinct vertex labels with
nonnegative entries:
\[
(0,0,0), (1,0,0), (0,1,0), (0,0,1), (2,0,0), (1,0,1), (0,2,0), (0,1,1).
\]
The surjectivity of the map $\psi$ in Theorem~\ref{main} ultimately implies,
through Theorem~\ref{thm:affine labeling}, that appending~$-1$ for the sink,
$q$, to each of these labels yields the set of $G$-parking functions (with
respect to $q$).  As they are, these labels are exactly the superstable
configurations for the abelian sandpile model on~$G$.  Subtracting each from the
maximal stable configuration, $(2,2,1)$, gives the recurrent configurations,
i.e., the elements of the sandpile group:
\[
(2,2,1), (1,2,1), (2,1,1), (2,2,0), (0,2,1), (1,2,0), (2,0,1), (2,1,0).
\]

Now pick any region $r$ and a point $t=(t_1,t_2,t_3)\in r$.  The point
determines a collection of closed intervals, $I_i=[t_i,t_i+1]$.  Turn the set of
intervals into a poset, $P$, by saying $I_i<I_j$ if $I_i$ lies completely to the left
of $I_j$ with no overlap, i.e., if $t_i+1<t_j$.  Overlapping intervals are not
comparable in $P$.   Posets arising from finite sets of intervals, ordered in
this way, are called {\em semiorders}.  Identifying interval $I_i$ with vertex
$v_i$ gives a poset on the nonsink vertices, and we then set $q$ to be the
unique minimal element to get a poset on all of the vertices.

For instance, the point $p=(1,3,2.5)$ is in the unbounded region directly to
the right of the central region in Figure~\ref{fig:example}, giving the
collection of intervals:
\[
\includegraphics[width=3in]{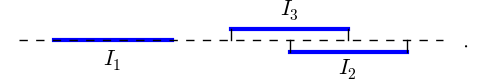}
\]
The Hasse diagram for the resulting semiorder on the vertices of $G$ is shown in
Figure~\ref{fig:semiorder example}.
\begin{figure}[h]
\begin{tikzpicture}
\SetVertexMath
\GraphInit[vstyle=Art]
\SetUpVertex[MinSize=3pt]
\SetVertexLabel
\tikzset{VertexStyle/.style = {%
shape = circle,
shading = ball,
ball color = black,
inner sep = 2pt
}}
\SetUpEdge[color=black]
\Vertex[LabelOut,Lpos=180,
Ldist=.1cm,x=-0.75,y=2]{v_2}
\Vertex[LabelOut,Lpos=0,
Ldist=.1cm,x=0.75,y=2]{v_3}
\Vertex[LabelOut,Lpos=180,
Ldist=.1cm,x=0,y=1]{v_1}
\Vertex[LabelOut,Lpos=180,
Ldist=.1cm,x=0,y=0]{q}
\Edges(v_3,v_1,v_2)
\Edges(q,v_1)
\end{tikzpicture}
\caption{Semiorder determined by $(1,3,2.5)\in\R^3$.}\label{fig:semiorder example}
\end{figure}
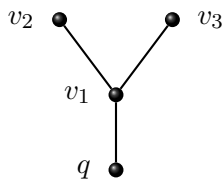

The partial orientation of the label for region $r$ can then be read from the
semiorder: $(v_i,v_j)$ appears as an oriented edge if and only if $\{v_i,v_j\}$
is an edge of $G$ and $v_i<v_j$ in $P$.  In general, varying the point $p$
selected in~$r$ may result in different semiorders: they may disagree for pairs
of vertices that do not determine an edge of $G$.  (In our example, each pair of
the $v_i$ form an edge, so only one poset on the vertices arises from each
region.) 

How does this example generalize to an arbitrary graph, $G$?  Suppose~$G$ has
designated sink $q$ and has nonsink vertices $v_1,\dots,v_n$.  Form the
arrangement of hyperplanes, $x_i-x_j=1$ for all $i\neq j$ such that
$\{v_i,v_j\}$ is an edge of $G$.  Thus, for each edge, $\{v_i,v_j\}$, there is a
``stripe'' consisting of the two hyperplanes $x_i-x_j=\pm1$.  The partial
orientations in our labeling record on which side of each stripe a given region
lies.

Label the central region, for which $|x_i-x_j|<1$ for all edges $\{v_i,v_j\}$,
by the graph $G$, orienting the edges incident with $q$ so that they point away
from $q$.  Label the $v_i$ by the number of oriented edges pointing into~$v_i$
minus~$1$ (hence, by $0$ or $-1$), and label~$q$ with~$-1$.  Then proceed
inductively to label all of the regions, as in the example.  In the end, the
regions will be in bijection with those partial orientations satisfying the
property that for each potential directed cycle, there are more blanks than
arrows.  The collection of vertex labels with nonnegative values at each $v_i$
are the $G$-parking functions, which correspond to the elements of the sandpile
group.  Picking a point in any region determines a semiorder on the vertices of
$G$, from which one may reconstruct the labeling of $G$ for the region.

We now describe a version of the above construction that avoids an initial
choice of a sink. Let $G$ be a graph with vertices $\{v_0,\dots,v_n\}$.  Form
the hyperplane arrangement as described above: $x_i-x_j=1$ for each $i\neq j$
such that $\{v_i,v_j\}$ is an edge of $G$.  Thus, there are two hyperplanes for
each edge.  Label the central region---given by $|x_i-x_j|<1$ for all $i,j$ such
that $\{i,j\}$ is an edge of $G$---with a copy of $G$ having no oriented edges and
with a $-1$ at each vertex.  Proceed as before, labeling each region.  By
Theorem~\ref{parking functions}, those regions for which $v_i$ has label
$-1$, and all other vertices have nonnegative labels, are the $G$-parking
functions with respect to~$v_i$.

For example, again take $G$ to be the graph in Figure~\ref{fig:graph} but with
no vertex chosen as sink (take $q=v_0$).  The corresponding hyperplane
arrangement consists of ten hyperplanes in $\R^4$.  Each hyperplane is parallel
to the vector $(1,1,1,1)$, so we project onto the hyperplane given by
$\sum_{i=0}^3x_i=0$, which we identify with $\R^3$, to get an arrangement whose
$109$ regions are in bijection with those of the original arrangement.  The
central region is a polytope with ten faces.  It is pictured in
Figure~\ref{fig:center} along with two of its bordering regions.  The bordering
region forming a pyramid on top of the central region is labeled by a copy of
$G$ with one directed edge, $(v_2,v_1)$, and with vertex label $(-1,0,-1,-1)$.
The other bordering region in the figure is labeled by a copy of $G$ with one
directed edge, $(v_3,v_1)$, and with the same vertex label, $(-1, 0,-1,-1)$.
\begin{figure}
  \includegraphics[width=2in]{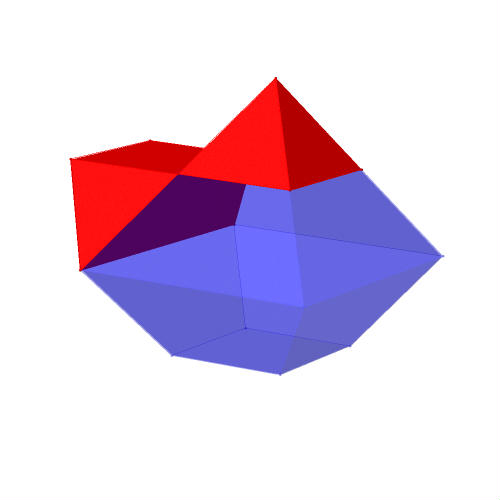}
  \caption{The central region and two bordering regions for
  $G$.}\label{fig:center}
\end{figure}
Figure~\ref{fig:choice} depicts those (unbounded) regions of the hyperplane
arrangement for $G$ that satisfy $x_i>x_0+1$ for $i=1,2,3$. They are in
bijection with the regions in Figure~\ref{fig:example}, and corresponding
regions would have the same labels.  (Recall that for convenience the drawing of
the sink vertex and its edges is suppressed in Figure~\ref{fig:example}.)
\begin{figure}
  \includegraphics[width=3in]{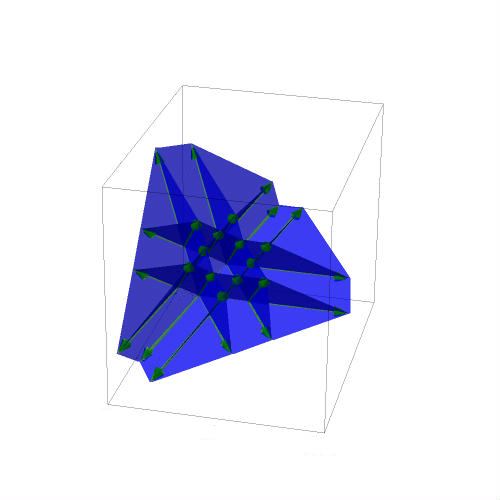}
  \caption{Regions for $G$ corresponding to choosing a sink vertex.}\label{fig:choice}
\end{figure}

\section{Four structures associated with \texorpdfstring{$G$}{G}}\label{section:structures}
From now on, we take $G$ to be a finite, connected, undirected graph,
with vertices $V=\{v_0,\dots,v_n\}$ and edges $E$.  Loops and multiple edges are
disallowed (the former for convenience of notation).

\subsection{\texorpdfstring{$G$}{G}-semiorientations.}  In this section we define $\mathbb{O}$, the collection of $G$-semiorientations.
A {\em partial orientation} of~$G$ is a choice of directions for a subset of the
edges of $G$.  Formally, a partial orientation is a subset $\mathcal{O}\subset
V\times V$ with the property that if $(u,v)\in\mathcal{O}$, then $\{u,v\}\in E$ and
$(v,u)\notin\mathcal{O}$.  Let~$\mathcal{O}$ be a partial orientation.  If
$e=\{u,v\}\in E$ and $(u,v)\in\mathcal{O}$, then despite the ambiguity, we write
$e\in\mathcal{O}$ and say $e$ is oriented.  In that case, we think of~$e$ as an arrow from $u$ to $v$
and write $e^{-}=u$ and $e^{+}=v$.  If neither $(u,v)$ nor $(v,u)$ is in
$\mathcal{O}$, we write $e\notin\mathcal{O}$ and say that $e$ is an unoriented
or {\em blank} edge.  The {\em outdegree} of the vertex $u\in V$ relative to
$\mathcal{O}$, denoted $\mathrm{outdeg}_{\mathcal{O}}(u)$, is the number of
edges $e\in\mathcal{O}$ such that $e^{-}=u$.  Similiarly, the {\em indegree} of
$u\in V$ relative to $\mathcal{O}$, denoted $\mathrm{indeg}_{\mathcal{O}}(u)$, is
the number of edges $e\in\mathcal{O}$ such that $e^{+}=u$.  We use the notation
$\mathrm{deg}(u)$ to denote the ordinary degree of $u$, i.e., the number of $e\in
E$ containing $u$.  Fixing $\mathcal{O}$, some of the edges of any cycle
$C\subseteq E$ of $G$ will be oriented and others will be blank.  If it is
possible to assign directions to the blank edges so that $C$ would become a
directed cycle, then we call~$C$ a {\em potential cycle} for $\mathcal{O}$.

\begin{definition}\label{semiorientation}
  A {\em $G$-semiorientation} is a partial orientation, $\mathcal{O}$, such that each potential cycle for $\mathcal{O}$ has more blank edges than oriented edges. The set of $G$-semiorientations of $G$ is denoted $\mathbb{O}$.
\end{definition}

\subsection{\texorpdfstring{$G$}{G}-semiorders.}\label{semiorder}  A reference for ordinary
semiorders is \cite{stanley}.  Let $k$ be any nonnegative integer, and consider
a collection of unit length closed intervals, $P=\{I_1,\dots,I_k\}$, of the real
line.  Order the elements of $P$ by $I_i<I_j$ if~$I_i$ lies strictly to the left
of~$I_j$, i.e., if $I_i=[a_i,a_i+1]$ and $I_j=[a_j,a_j+1]$, then $a_i+1<a_j$.
Any poset isomorphic to a poset~$P$, constructed as above, is called a {\em
semiorder}.  The number of non-isomorphic semiorders with~$k$ elements is the
$k$-th Catalan number,~$C_k$, and there is a corresponding generating function
\[
C(x)=\sum_{k\geq0}C_nx^k=\frac{1-\sqrt{1-4x}}{2x}.
\]
If $f_k$ denotes the number of {\em labeled} semiorders,  there is the exponential generating function
\[
\sum_{k\geq0}f_k\frac{x^k}{k!}=C(1-e^{-x}).
\]

\begin{definition}
  A {\em $G$-semiorder} is a semiorder on the vertices of $G$. The set of
  $G$-semiorders is denoted $\mathbb{I}$.  
\end{definition}

\subsection{The \texorpdfstring{$G$}{G}-semiorder arrangement.}
\begin{definition}\label{semiorder arrangement}
The {\em $G$-semiorder arrangement}, denoted $\mathscr{I}$, is the set of
$2|E|$ hyperplanes in $\R^{n+1}$ given by
\[
x_i-x_j=1,
\]
for all $i\neq j$ such that $\{v_i, v_j\}\in E$.    The {\em regions} of
$\mathscr{I}$, denoted $\mathcal{R}$, are the connected components of
$\R^{n+1}\setminus\mathscr{I}$.
\end{definition}
If $G$ were the complete graph on $n+1$ vertices, the $G$-semiorder arrangement
would be the ordinary semiorder arrangement discussed in \cite{stanley}, whose
regions are in bijection with labeled semiorders on $n+1$ elements.  For
general~$G$, each region of the $G$-semiorder arrangement is a union of regions
from the ordinary semiorder arrangement.

\subsection{Quasi-superstables on \texorpdfstring{$G$}{G}.}\label{quasi-superstables} Now
designate $q:=v_0$ as the {\em sink} vertex.  We recall the basic facts about
the abelian sandpile model on~$G$, including $G$-parking functions.  A reference
for the sandpile results stated here is~\cite{HLPW}. (More references for the
abelian sandpile model:~\cite{BTW} and~\cite{dhar} are seminal;~\cite{dhar2} and
\cite{PW} are general references; and~\cite{LP} is a quick
overview.)  We then proceed to  define $\mathcal{S}$, the set of
quasi-superstable divisors of~$G$.

Let $D=\mathrm{diag}(\mathrm{deg}\,v_0,\dots,\mathrm{deg}\,v_n)$, and let $A$ be
the adjacency matrix for~$G$, defined by 
\[
A_{ij}=
\begin{cases}
  1&\text{if $\{v_i,v_j\}\in E$},\\
 \hfill 0&\text{otherwise}.
\end{cases}
\]
The {\em Laplacian} matrix for $G$ is
\[
\Delta = D-A.
\]
The {\em reduced Laplacian} is the matrix $\widetilde{\Delta}$ obtained by removing
the first row and column of $\Delta$, i.e., the row and column corresponding to
the sink vertex.

A {\em configuration} on $G$ is an element of the free abelian group on the
nonsink vertices.  Having ordered the nonsink vertices, as above, we identify
the set of configurations with $\Z^n$ in the natural way:
$c=\sum_{i=1}^nc_i\,v_i\leftrightarrow(c_1,\dots,c_n)$.  A {\em divisor} on $G$
is an element of the free abelian group on all of the vertices of~$G$, which we
similarly identify with $\Z^{n+1}$.  Once a sink is chosen, we may consider
configurations as those divisors whose sink coefficient is $0$.

Given two configurations or two divisors $c$ and $c'$, we write $c\geq c'$ if
$c_i\geq c'_i$ for all $i$.  We say~$c$ is {\em nonnegative} if $c\geq 0$, i.e.,
if each component of~$c$ is nonnegative.  

If~$X$ is a subset of the nonsink vertices, write $1_X$ for the configuration
whose $i$-th component is $1$ if $v_i\in X$ and $0$ otherwise.  {\em Firing} $X$
from a configuration~$c$ results in the configuration
$c-\widetilde{\Delta}\,1_X$.  If $X=\{v_i\}$, we call this operation {\em firing
$v_i$}.  One may speak of {\em firing the sink vertex}, which adds~$1$ to each
vertex connected to the sink.  Since the sum of the columns of the Laplacian
matrix is zero, firing the set of all nonsink vertices is the same as {\em
reverse-firing} the sink, subtracting $1$ from each vertex attached to the
sink.

\begin{definition}
Let $c$ be a nonnegative configuration on $G$.  
We say $c$ is {\em stable} if there is no $i$ such that firing~$v_i$ from $c$
results in a nonnegative configuration.   We say that $c$ is {\em superstable} if
there is no nonempty subset~$X$ of the nonsink vertices such that firing $X$
from $c$ results in a nonnegative configuration. 
\end{definition}

The notion of a superstable configuration is essentially the same as that of a
parking function.
\begin{definition}
  A function $f\colon V\to\Z$ is a {\em $G$-parking function} (with respect to
  $q$) if $f(q)=-1$ and $(f(v_1),\dots,f(v_n))$ is a superstable configuration
  on~$G$.  We identify a $G$-parking function $f$ with the divisor
  $\sum_{v\in V}f(v)\,v$.
\end{definition}

Suppose $c$ is a nonnegative configuration.  Then $c$ is stable exactly when
$c_i<\mathrm{deg}\,v_i$ for all $i$.  A vertex~$v_i$ such
that~$c_i<\mathrm{deg}\,v_i$ is said to be {\em stable} in~$c$; otherwise it is
{\em unstable}.  Firing a set of vertices is {\em legal} from $c$ if the
resulting configuration is nonnegative.  In particular, firing a single unstable
vertex of~$c$ is legal.  A sequence of vertices is called a {\em legal firing
sequence} for~$c$ if each vertex in the sequence is unstable after firing the previous
vertices in the sequence.  Since there is a path from each nonsink vertex to the
sink in~$G$, there is a legal firing sequence leading to a stable configuration
$c^{\circ}$ called the {\em stabilization of~$c$}.  This process is called {\em
stabilizing $c$}.  It turns out that~$c^{\circ}$ is independent of the order in
which unstable vertices are fired, as is the number of times each vertex is fired
in reaching $c^{\circ}$.    
\begin{definition}
  A stable configuration $c\geq 0$ is {\em recurrent} if given any nonnegative
  configuration $a$, there exists a nonnegative configuration $b$ such that
  $c=(a+b)^{\circ}$.  The recurrent elements with the operation of (vertex-wise)
  addition followed by stabilization is called the {\em sandpile group} of $G$
  (with respect to $q$), denoted $\mathrm{Sand}(G)$.
\end{definition}

It is well-known that the sandpile group actually is a group and the mapping
\begin{align*}
  \mathrm{Sand}(G)&\to\Z^n/\mathrm{image}(\widetilde{\Delta})\\
  c&\mapsto c
\end{align*}
is an isomorphism.  Each equivalence class of $\Z^n$ modulo the image of the
reduced Laplacian contains a unique recurrent element.  It is also known that
each equivalence class contains a unique superstable element.  Define the
{\em maximal stable configuration} to be
\[
c_{\mathrm{max}}=\sum_{i=1}^n(\mathrm{deg}\,v_i-1)\,v_i.
\]
The next two propositions are well-known.
\begin{prop}[{\cite[Theorem~4.4]{HLPW}}]\label{duality}
The configuration $c$ is recurrent if and only if $c_{\mathrm{max}}-c$ is superstable.
\end{prop}

\begin{prop}[{Dhar's burning algorithm, \cite{dhar2}, \cite[Lemma~4.1]{HLPW}}]\label{dhar}
  Let $b\geq0$ be a stable configuration on $G$, and let $\tilde{b}$ be the
  configuration obtained from~$b$ by firing the sink.  The following are
  equivalent:
  \begin{enumerate}
    \item $b$ is recurrent,
    \item $(\tilde{b})^{\circ}=b$, i.e., the stabilization of $\tilde{b}$ is $b$,
    \item in stabilizing $\tilde{b}$, each nonsink vertex fires exactly once.
  \end{enumerate}
\end{prop}
\noindent {\em Remark.} We refer to Proposition~\ref{dhar} as Dhar's burning
algorithm, although it would more properly be called the theoretical
underpinning of the algorithm.
\begin{definition}
Let $K(G)$ be the graph $G$ with the addition of a new vertex~$\tilde{q}$ and
edges $\{\tilde{q},v_i\}$ for all vertices $v_i$, including $q=v_0$.  Set
$\tilde{q}$ as the sink vertex of~$K(G)$.  
\end{definition}
Thus, the set of nonsink vertices of $K(G)$ is $V$, and each of these is
connected to the sink, $\tilde{q}$, by an edge.  The divisors on $G$
are exactly the configurations on  $K(G)$.
\begin{definition}\label{def:quasi-superstable}
A divisor $c\in\Z^{n+1}$ on $G$ is called {\em quasi-superstable} if $c=
\tilde{c}-1_V$
for some superstable $\tilde{c}$ on $K(G)$.  The collection of quasi-superstable
divisors is denoted $\mathcal{S}$.
\end{definition}
For the relation between the superstables and quasi-superstables of $G$, see
Theorem~\ref{parking functions}~(\ref{ss/qss2}).

\section{Correspondences between the structures}\label{section:correspondences}
So far, we have defined the following structures on $G$:
\begin{align*}
  \mathbb{O}&:\quad\text{$G$-semiorientations},\\
  \mathbb{I}&:\quad\text{$G$-semiorders},\\
  \mathscr{I}&:\quad\text{the $G$-semiorder arrangement},\\
  \mathcal{S}&:\quad\text{quasi-superstable divisors on $G$}.
\end{align*}
In this section, we describe relations among these structures, culminating in
Theorems~\ref{main} and~\ref{thm:labeling}. 
\subsection{Semiorders and semiorientations}\label{semiorders and
semiorientations}
 \begin{definition}
   A $G$-semiorder $P$ and a $G$-semiorientation $\mathcal{O}$ are {\em
   compatible} if for each edge $e=\{u,v\}$ of $G$, we have that $u<v$ if and
   only if $(u,v)\in\mathcal{O}$.  Thus, if $e\notin\mathcal{O}$, then $u$ and
   $v$ are not comparable in $P$.
\end{definition}
Given a $G$-semiorder, $P$, define
\[
\mathcal{O}_P=\{(u,v): \text{$\{u,v\}\in E$ and $u<v$ in $P$}\}.
\]
\begin{thm}\label{compatibility}
  Let $P$ be a $G$-semiorder.  Then $\mathcal{O}_P\in\mathbb{O}$ and  $\mathcal{O}_P$ is the unique element of $\mathbb{O}$ compatible with $P$.
\end{thm}
\begin{proof}
The only part of this theorem that is not immediate from the definitions is the
fact that every potential cycle for $\mathcal{O}_P$ has more blanks edges than
oriented edges, which we now prove. 
Let $\alpha=\{e_1,\dots,e_k\}$ be a potential cycle.  We may assume that
$e_i=\{u_i,u_{i+1}\}$ where $u_{k+1}=u_1$, and that for each $i$, either
(i) $(u_i,u_{i+1})\in\mathcal{O}_P$, in which case $u_i<u_{i+1}$, or (ii)~$e_i$
is a blank edge, in which case $u_i$ and $u_{i+1}$ are not comparable. Since $P$
is a semiorder, it is isomorphic to a semiorder of intervals, allowing us to
identify each $u_i$ with an interval $I_i=[a_i,a_i+1]$.  If
$e_i\in\mathcal{O}_P$, we have $I_i<I_{i+1}$, in which case $a_i<a_{i+1}-1$; and
if $e_i\notin\mathcal{O}_P$, then $I_i$ and $I_{i+1}$ overlap, so in particular,
$a_i\leq a_{i+1}+1$.
Thus, if there are
$\delta$ oriented edges and $\beta$ blank edges in $\alpha$, since
$u_{k+1}=u_1$, 
\[
a_1+\delta -\beta\leq a_1,
\]
with equality if and only if $\beta=\delta=0$.  However, since $\alpha$ has at
least one edge, the inequality must be strict, and $\beta>\delta$ as required.
\end{proof}

Thus, we can refer to {\em the} semiorientation determined by a
$G$-semiorder and define the mapping
\begin{align*}
  \nu: \mathbb{I}&\to\mathbb{O}\\
\nonumber P&\mapsto\mathcal{O}_P.
\end{align*}
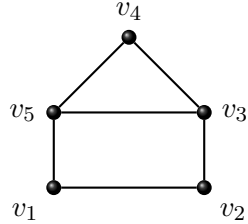
\begin{figure}[ht] 
\begin{tikzpicture}
\SetVertexMath
\GraphInit[vstyle=Art]
\SetUpVertex[MinSize=3pt]
\SetVertexLabel
\tikzset{VertexStyle/.style = {%
shape = circle,
shading = ball,
ball color = black,
inner sep = 2pt
}}
\SetUpEdge[color=black]
\Vertex[LabelOut,Lpos=225,x=0,y=0]{v_1}
\Vertex[LabelOut,Lpos=315,x=2,y=0]{v_2}
\Vertex[LabelOut,Lpos=0,x=2,y=1]{v_3}
\Vertex[LabelOut,Lpos=90,x=1,y=2]{v_4}
\Vertex[LabelOut,Lpos=180,x=0,y=1]{v_5}
\Edges(v_1,v_2,v_3,v_4,v_5,v_1)
\Edges(v_3,v_5)
\end{tikzpicture}
\caption{The house graph.}\label{fig:house graph}
\end{figure}
\begin{example} Let $G$ be the {\em house graph} of Figure~\ref{fig:house graph}.
  Figure~\ref{fig:compatible semiorientation} depicts a $G$-semiorder and its
  corresponding compatible $G$-semiorientation.
\end{example}

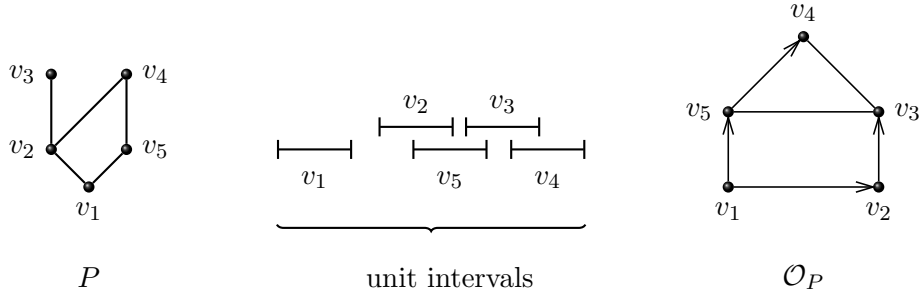
\begin{figure} 
\begin{tikzpicture}
\SetVertexMath
\GraphInit[vstyle=Art]
\SetUpVertex[MinSize=3pt]
\tikzset{VertexStyle/.style = {%
shape = circle,
shading = ball,
ball color = black,
inner sep = 1.5pt
}}
\SetVertexLabel
\SetUpEdge[color=black]

\Vertex[LabelOut,Lpos=180,x=0,y=0]{v_2}
\Vertex[LabelOut,Lpos=180,x=0,y=1]{v_3}
\Vertex[LabelOut,Lpos=0,x=1,y=0]{v_5}
\Vertex[LabelOut,Lpos=0,x=1,y=1]{v_4}
\Vertex[LabelOut,Lpos=270,x=0.5,y=-0.5]{v_1}
\Edges(v_3,v_2,v_1,v_5,v_4)
\Edge(v_4)(v_2)
\draw (0.5,-1.7) node{$P$};

\def\i{3.5}
\draw[{|-|},thick] (\i-0.5,0) -- (\i+1-0.5,0);
\draw[{|-|},thick] (\i+1.3,0) -- (\i+2.3,0);
\draw[{|-|},thick] (\i+2.6,0) -- (\i+3.6,0);
\draw[{|-|},thick] (\i+0.85,0.3) -- (\i+1.85,0.3);
\draw[{|-|},thick] (\i+2.0,0.3) -- (\i+3.0,0.3);
\draw (\i+0.5-0.5,-0.4) node{$v_1$};
\draw (\i+1.8,-0.4) node{$v_5$};
\draw (\i+3.1,-0.4) node{$v_4$};
\draw (\i+1.35,0.6) node{$v_2$};
\draw (\i+2.5,0.6) node{$v_3$};
\draw[decoration={brace,mirror},decorate,thick] (\i-0.5,-1.0) -- (\i+3.6,-1.0);
\draw (\i+1.8,-1.7) node{unit intervals};

\def\j{\i+5.5}
\def\k{-0.5}
\Vertex[LabelOut,Lpos=90,x=\j+1,y=2+\k]{v_4}
\Vertex[LabelOut,Lpos=180,x=\j+0,y=1+\k]{v_5}
\Vertex[LabelOut,Lpos=270,x=\j+0,y=0+\k]{v_1}
\Vertex[LabelOut,Lpos=270,x=\j+2,y=0+\k]{v_2}
\Vertex[LabelOut,Lpos=0,x=\j+2,y=1+\k]{v_3}
\Edge[style={->,>=mytip, semithick}](v_5)(v_4)
\Edge[style={->,>=mytip, semithick}](v_1)(v_2)
\Edge[style={->,>=mytip, semithick}](v_1)(v_5)
\Edge[style={->,>=mytip, semithick}](v_2)(v_3)
\Edge[style={semithick}](v_4)(v_3)
\Edge[style={semithick}](v_5)(v_3)
\draw (\j+1,-1.7) node{$\mathcal{O}_P$};
\end{tikzpicture}
\caption{A $G$-semiorder $P$ on the vertices of the house graph, $G$; a
collection of intervals realizing the semiorder; and the unique
$G$-semiorientation, $\nu(P)=\mathcal{O}_P$, compatible with
$P$.}\label{fig:compatible semiorientation}
\end{figure}

\subsection{A bijection between semiorientations and hyperplane regions}\label{hyperplane bijection}
We now define a mapping 
\[
\rho\colon\mathbb{O}\to\mathcal{R}.  
\] 
If $\mathcal{O}\in\mathbb{O}$, let $\rho(\mathcal{O})$ be the region
defined by the following inequalities: for each edge $e$ of $G$: 
\begin{itemize} 
  \item if $e=(v_i,v_j)\in\mathcal{O}$, then $x_j > x_i + 1$,
  \item if $e\notin\mathcal{O}$, then $|x_i-x_j|<1$.  
\end{itemize}

The reader may find it helpful to skim the example below the proof of the
following theorem before reading the proof itself.
\begin{thm}\label{hyperplane bijection theorem}
The mapping $\rho$ is a well-defined bijection.
\end{thm}

\begin{proof}
  Let $\mathcal{O}\in\mathbb{O}$, and let $r=\rho(\mathcal{O})$.
The system of inequalities defining $r$ indicates on which side of each hyperplane
of $\mathscr{I}$ the region sits.  To show $r$ is a region of
$\mathscr{I}$, it suffices to show that $r$ is nonempty.

We define a directed, weighted graph $G'$ with the same vertices as $G$ in the
following manner: for each edge $e$ between vertices of $G$, 
\begin{itemize}
  \item if $e=(v_i,v_j)\in \mathcal{O}$, then $(v_i,v_j) \in G'$ and the weight
    of $(v_i,v_j)$ in $G'$ is $-1$,
\item if $e = \{v_i,v_j\}\notin \mathcal{O}$, then $(v_i,v_j), (v_j,v_i) \in G'$
  and the weight of both $(v_i,v_j)$ and $(v_j,v_i)$ in $G'$ is $1$.  
\end{itemize}

Choose an ordering of the edges of~$G'$: $e'_1, e'_2,\dots, e'_k$.  Define a
$k\times (n+1)$ edge-vertex adjacency matrix with rows $r_1,\dots, r_k$ as follows:
if $e'_{\ell} = (v_i, v_j)$, let $r_{\ell}$ be the vector having $1$ in the
$i$th entry, $-1$ in the $j$th entry, and $0$s elsewhere.  Let $b$ be a column
vector in $\R^{k}$ where $b_{\ell}$ is the weight of $e'_{\ell}$.

Thus, the inequalities of $\rho(\mathcal{O})$ are encoded as $Ax < b$. 
By Farkas' lemma the insolvability of $Ax<b$ is equivalent to the existence
of a row vector $y = (y_1, ..., y_k)$ satisfying:
\begin{equation}\label{farkas}
y_i\geq0\ \ \forall i,\quad y\neq0,\quad yA=0,\quad y\cdot b\leq0. 
\end{equation}
For sake of contradiction, suppose such a $y$ exists.
The {\em support} of $y$ is
\[
\mathrm{supp}(y)=\{i: y_i\neq0\}.
\]
Among all row vectors satisfying condition~(\ref{farkas}), suppose $y$ has been
chosen so that the cardinality of its support is minimal.
Say $\ell_1\in\mathrm{supp}(y)$ and $e'_{\ell_1}=(v_i,v_j)$.  Hence,
$r_{\ell_1}$ has a $-1$ in its $v_j$-th entry. Since
\[
yA = y_1\,r_1\dots+y_k\,r_k=0,
\]
and the components of $y$ are nonnegative, there must be some
$\ell_2\in\mathrm{supp}(y)$ such that $r_{\ell_2}$ has a~$1$ in its $v_j$-th entry.
This row will have a $-1$ in some other entry, forcing the existence of some
$\ell_3\in\mathrm{supp}(y)$ such that $r_{\ell_3}$ has a~$1$ in that entry, and
so on.
Since the support of $y$ is finite, the sequence $\ell_1,\ell_2,\dots,$
eventually has a repeat.  Thus, there is a sequence of elements
$j_1:=\ell_{m+1},j_2:=\ell_{m+2},\dots,j_t:=\ell_{m+t}$ in the support
of $y$ for some $t$ corresponding to a directed cycle of edges $e'_{j_1},\dots,e'_{j_t}$ in $G'$.

Let $z=(z_1,\dots,z_k)$ be the row vector with $z_{\ell}=1$ if $\ell\in\{j_1,\dots,j_t\}$ and
$z_{\ell}=0$, otherwise.
Since the support of $z$ corresponds to a directed cycle of edges in $G'$, we
have $zA = 0$. Furthermore, since any potential cycle in~$\rho(\mathcal{O})$ has
more blank edges than oriented edges, we have $z\cdot b > 0$.  Let
$a=\min\{y_{j_1},\dots,y_{j_t}\}$ and define $y'=y-a z$.  Then $y'$
satisfies condition~(\ref{farkas}) but its support is strictly contained in the
support of $y$, yielding a contradiction.  So there must
be some solution to $Ax<b$, which means that $r=\rho(\mathcal{O})\in\mathcal{R}$.

We now define a mapping
\[
\tau:\mathcal{R}\to\mathbb{O}.
\]
If $r\in\mathcal{R}$ let $\tau(r)$ be the partial orientation of $G$ obtained by
(i) $(v_i,v_j)\in\tau(r)$ if $\{v_i,v_j\}\in E$ and $x_j>x_i+1$ in $r$, and (ii)
all other edges of $G$ are blank.  Once we show~$\tau$ is well-defined, it is
immediate that it is the inverse of $\rho$.

Let $r\in\mathcal{R}$. To see that $\tau(r)\in\mathbb{O}$, suppose $\tau(r)$
has a potential cycle~$\alpha$ having at least as many oriented edges as blank
edges.  Define $A$ and~$b$ as above to encode the system of inequalities
that defines the region $r$ as $Ax < b$.  Let~$y$ be the row vector with $1$s in
the entries corresponding to the (oriented and blank) edges of~$\alpha$ and
$0$s elsewhere.  We have $y \geq 0$, $y \neq 0$, $yA = 0$, and $y \cdot b \leq
0$.  But by Farkas' lemma this means $Ax < b$ has no solutions, contradicting
the fact that $r\in\mathcal{R}$.  Thus, $\tau(r) \in \mathbb{O}$.
\end{proof}

\begin{example}  This example is intended to be read in conjunction with the
  proof of Theorem~\ref{hyperplane bijection theorem} and illustrates Farkas'
  lemma at work.  Consider the partial orientation, $\mathcal{O}$, of a triangle
  pictured in Figure~\ref{fig:triangle}.
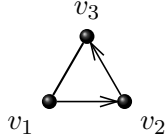
\begin{figure}[ht] 
\begin{tikzpicture}
\SetVertexMath
\GraphInit[vstyle=Art]
\SetUpVertex[MinSize=3pt]
\SetVertexLabel
\tikzset{VertexStyle/.style = {%
shape = circle,
shading = ball,
ball color = black,
inner sep = 2pt
}}
\SetUpEdge[color=black]
\Vertex[LabelOut,Lpos=225,x=0,y=0]{v_1}
\Vertex[LabelOut,Lpos=315,x=1,y=0]{v_2}
\Vertex[LabelOut,Lpos=90,x=0.5,y=0.866]{v_3}
\Edge[style={->,>=mytip,semithick}](v_1)(v_2)
\Edge[style={->,>=mytip,semithick}](v_2)(v_3)
\Edge(v_3)(v_1)
\end{tikzpicture}
\caption{Partial orientation, $\mathcal{O}$.}\label{fig:triangle}
\end{figure}
Naively attemtping to apply the region-labeling function,~$\rho$, to
$\mathcal{O}$ results in the system of inequalities $Ax<b$ where
\[
A =
    \left[
    \begin{array}{rrr}
      1&-1&0\\
      0&1&-1\\
      1&0&-1\\
      -1&0&1
    \end{array}
    \right],
\quad
x = 
    \left[
    \begin{array}{r}
      x_1\\x_2\\x_3
    \end{array}
    \right],
\quad
b=
    \left[
    \begin{array}{r}
     -1\\-1\\1\\1
    \end{array}
    \right].
\]
The rows of $A$ correspond to directed edges, $(v_1,v_2)$, $(v_2,v_3)$,
$(v_1,v_3)$, and $(v_3,v_1)$, respectively, of the underlying triangle,~$G$.
The last two inequalities combine to describe the ``sandwich'' $|x_1-x_3|<1$,
corresponding to the undirected edge, $\{v_1,v_3\}$.  

The partial orientation, $\mathcal{O}$, has a potential cycle with more oriented
edges than blanks, and hence is not a $G$-semiorientation.  By the way in which~$A$ is
defined in the proof of Theorem~\ref{hyperplane bijection theorem}, we know that
adding the rows of~$A$ corresponding to the directed cycle $(v_1,v_2)$,
$(v_2,v_3)$, $(v_3,v_1)$ gives the zero vector.  Our system of inequalities is
inconsistent since adding the corresponding entries of $b$ gives $-1$ (since
there are more oriented edges than blanks in $\mathcal{O}$ for this cycle).
\end{example}

Let $r$ be a region in the image of $\rho$. If $t=(t_0,\dots,t_n)\in r$, define
the unit intervals $I_i=[t_i,t_i+1]$ for $i=0,\dots,n$. Define $P_t$ to be the
$G$-semiorder determined by these intervals, labeled by the vertices of $G$ by
identifying $I_i$ with $v_i$.

\begin{thm}\label{region-semiorder}
  Let $\mathcal{O}\in\mathbb{O}$, and let
  $r=\rho(\mathcal{O})$.  The semiorders $P_t$ as~$t$ ranges
  over points in $r$ are exactly the $G$-semiorders compatible with
  $\mathcal{O}$.  
\end{thm}
\begin{proof}
Choose any $t\in r$.  Say $(v_i,v_j)\in\mathcal{O}$.  Then $x_j>x_i+1$
in $r$; so $t_j>t_i+1$, and hence, $v_i < v_j$ in $P_t$.  Now suppose $e =
\{v_i,v_j\}\in E$ but $e\notin\mathcal{O}$.  Then $|t_i-t_j|<1$, which means
that $I_i$ and $I_j$ overlap, and hence, $v_i$ and $v_j$ are not comparable in
$P_t$.  This shows that $P_t$ is compatible with $\mathcal{O}$.

Now let $P$ be a $G$-semiorder compatible with $\mathcal{O}$.  The semiorder $P$
is isomorphic to the semiorder on a set of unit intervals, $\{I_i\}_{i=0}^n$,
where $I_i$ corresponds to $v_i$. Say $I_i = [t_i,t_i+1]$ for each $i$, and
let $t = (t_0,\dots,t_i)$.  So $P=P_t$, but we must show that $t\in r$. Suppose
$e = \{v_i, v_j\} \in E$.  If $x_j > x_i + 1$ in $r$, then since
$\mathcal{O}=\tau(r)$, we have $(v_i,v_j)\in\mathcal{O}$, and thus $t_j>t_i+1$.
If $|x_i-x_j|<1$, then $e\notin\mathcal{O}$ and the intervals $I_i$ and $I_j$
overlap, i.e., $|t_i-t_j|<1$.  Hence, $t\in r$.
\end{proof}
\begin{example}  
  Let $G$ be the house graph pictured in Figure~\ref{fig:house graph}, and let 
  $\mathcal{O}$ be the $G$-semiorientation pictured on the right in
  Figure~\ref{fig:compatible semiorientation}.  The region
  $\rho(\mathcal{O})$ corresponding to $\mathcal{O}$ is determined by the system
  of inequalities: 
  \begin{align*}
   x_2&>x_1+1,\qquad x_3>x_2+1,\\
   x_4&>x_5+1,\qquad x_5>x_1+1,\\
  |x_3&-x_4|<1,\qquad|x_3-x_5|<1.
\end{align*}
Figure~\ref{fig:compatible semiorders} displays points
$t,t'\in\rho(\mathcal{O})$ whose corresponding semiorders, $P_t$ and
$P_{t'}$, are the two $G$-semiorders compatible with $\mathcal{O}$.
\end{example}
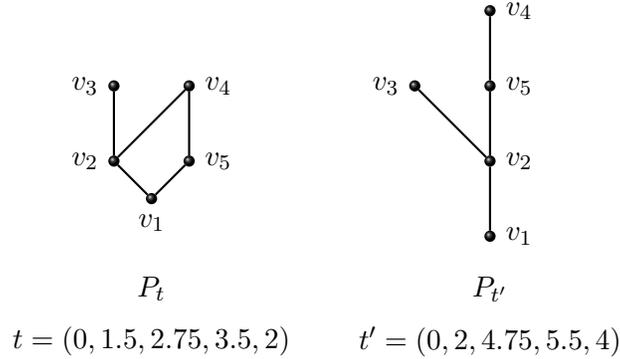
\begin{figure} 
\begin{tikzpicture}
\SetVertexMath
\GraphInit[vstyle=Art]
\SetUpVertex[MinSize=3pt]
\tikzset{VertexStyle/.style = {%
shape = circle,
shading = ball,
ball color = black,
inner sep = 1.5pt
}}
\SetVertexLabel
\SetUpEdge[color=black]

\Vertex[LabelOut,Lpos=180,x=0,y=0]{v_2}
\Vertex[LabelOut,Lpos=180,x=0,y=1]{v_3}
\Vertex[LabelOut,Lpos=0,x=1,y=0]{v_5}
\Vertex[LabelOut,Lpos=0,x=1,y=1]{v_4}
\Vertex[LabelOut,Lpos=270,x=0.5,y=-0.5]{v_1}
\Edges(v_3,v_2,v_1,v_5,v_4)
\Edge(v_4)(v_2)
\draw (0.5,-1.7) node{$P_t$};
\draw (0.5,-2.4) node{$t=(0,1.5,2.75,3.5,2)$};

\def\i{4}
\Vertex[LabelOut,Lpos=180,x=\i+0,y=1]{v_3}
\Vertex[LabelOut,Lpos=0,x=\i+1,y=-1]{v_1}
\Vertex[LabelOut,Lpos=0,x=\i+1,y=0]{v_2}
\Vertex[LabelOut,Lpos=0,x=\i+1,y=1]{v_5}
\Vertex[LabelOut,Lpos=0,x=\i+1,y=2]{v_4}
\Edges(v_4,v_5,v_2,v_1)
\Edge(v_2)(v_3)
\draw (\i+1.0,-1.7) node{$P_{t'}$};
\draw (\i+1.0,-2.4) node{$t'=(0,2,4.75,5.5,4)$};
\end{tikzpicture}
\caption{Two semiorders.}\label{fig:compatible semiorders}
\end{figure}

\subsection{Superstables algorithm}\label{ss-alg}
So far, we have described relations among three of our four structures: $G$-semiorders, 
$G$-semiorientations, and the regions of the $G$-semiorder arrangement.  We
would now like to forge a connection between these and the quasi-superstables on
$G$ (ultimately relating all of this, in Section~\ref{sink}, to ordinary
superstables).

Dhar's burning algorithm says that starting from a recurrent configuration,
then firing the sink, each non-sink vertex will fire exactly once in the
stabilization process.  We use these vertex firings to build a partial
orientation starting from the unoriented graph.  To sketch the idea, suppose that at
some point in the stabilization process there are nonsink vertices $v$ and~$w$
such that $e=\{v,w\}$ is an edge and $v$ is unstable.  Suppose $e$ has not
already been oriented or marked as a blank edge.  When $v$ fires, if $w$ is
stable, orient $e$ as $(v,w)$; otherwise mark $e$ blank.  Since all vertices
fire, each edge is visited. 

Starting with a quasi-superstable, $c$, on $G$, Proposition~\ref{duality}
provides a corresponding recurrent configuration on $K(G)$.  We use Dhar's
algorithm as above to create a $G$-semiorientation, $\mathcal{O}$, of $G$.
Using the construction of $\mathcal{O}$ as a guide, one may further refine the
procedure to simultaneously create a $G$-semiorder compatible with
$\mathcal{O}$.  This is the idea behind the intervals labeled by $J$ in the
algorithm, below. Theorem~\ref{main} then provides the important connection:
subtract one from each entry of the indegree sequence relative to $\mathcal{O}$
to recapture $c$.  

A careful description of our algorithm follows.  The input is a
quasi-superstable divisor on $G$ and a vertex ordering, and the output is a
$G$-semiorder $P$ and the $G$-semiorientation $\mathcal{O}_P$ compatible with
$P$.  Letting $\mathfrak{S}$ be the symmetric group on $\{0,\dots,n\}$, identify
$\sigma\in\mathfrak{S}$ with the vertex ordering
$(v_{\sigma(0)},\dots,v_{\sigma(n)})$.  Then our algorithm defines two mappings
$\phi:\mathcal{S}\times\mathfrak{S}\to\mathbb{I}$ and
$\eta:\mathcal{S}\times\mathfrak{S}\to\mathbb{O}$.  Theorem~\ref{main}, to
follow, validates the algorithm and shows that it produces a commutative diagram 
\begin{equation}\label{ss-so-po} 
\xymatrix{
\mathcal{S}\times\mathfrak{S}\ar[r]^-{\phi}\ar[rd]_{\eta}&\mathbb{I}\ar[d]^{\nu}\\
{}&\mathbb{O},  
} 
\end{equation}
where $\nu$ is the mapping defined in Section~\ref{semiorders and
semiorientations}.

Given a $c\in\mathcal{S}$ and a vertex ordering
$\sigma\in\mathfrak{S}$, the algorithm proceeds as
follows:
\needspace{5\baselineskip}
\bigskip

\noindent \fbox{\sc initialization} 
\medskip

Let $c_{\mathrm{max}}$ be the maximal stable configuration of $K(G)$, and let
\[
b = c_{\mathrm{max}}-c.
\]
Since $c\in\mathcal{S}$, we can write $c=\tilde{c}-1_V$ for some superstable
$\tilde{c}$ on
$K(G)$.  Therefore, $b=(c_{\mathrm{max}}-\tilde{c})+1_V$, i.e., $b$ is the
configuration on $K(G)$ obtained from $c_{\mathrm{max}}-\tilde{c}$ by firing the
sink $\tilde{q}$ of $K(G)$.  By Proposition~\ref{duality},
$c_{\mathrm{max}}-\tilde{c}$ is a recurrent configuration on $K(G)$; so by
Proposition~\ref{dhar}, every element of $V$ will fire exactly once while
stabilizing~$b$.

Let $u_1,\dots,u_k$ be the vertices that are unstable in $b$. Take this list of
vertices to be ordered according to $\sigma$, that is, if $u_i=v_{\sigma(\ell)}$
and $u_{j}=v_{\sigma(m)}$ with $\ell<m$, then $i<j$.  For each $u_i$,
associate an interval,
\[
J(u_i) = [i/(k+1), 1+i/(k+1)].
\]
Thus, all the $J(u_i)$ overlap. Form a queue,
\[
Q=(u_1,\dots,u_k).
\]
Initialize the partial orientation of $G$ as $\mathcal{O}=\emptyset.$
\smallskip

\noindent \fbox{\sc loop} 
\medskip

\noindent Repeat the following until the queue is empty:
\medskip

Say the queue is $Q=(w_0,\dots,w_{\ell})$, and 
   \[
   J(w_0)=[\alpha,\alpha+1].
   \]
   If $\ell>0$, define $\varepsilon$ using the interval for $w_1$:
   \[
    J(w_1)=[\alpha+\varepsilon,\alpha+\varepsilon+1].
   \]
   Otherwise, take $\varepsilon=1$.

   Fire $w_0$ and replace $b$ by the resulting configuration.  Remove $w_0$ from
   the queue and mark it so that it will never again appear in the queue.

   Let $z_1,\dots,z_t$ be the vertices that just became unstable with the firing
   of~$w_0$ and that have not yet been fired by the algorithm, listed in order
   according to $\sigma$.  Define
   \[
   J(z_i)=[\alpha+1+i\varepsilon/(t+1),\alpha+2+i\varepsilon/(t+1)].
   \]
  Add the $z_i$ to $Q$, in order:
  \[
  Q=(w_1,\dots,w_{\ell},z_1,\dots,z_t).
  \]

   For each edge $e=\{w_0,v\}$
   that is not oriented or marked as blank, 
 \begin{enumerate}
    \item if $v$ was already unstable before the firing of $w_0$,
      mark $e$ as blank, so~$e$ will not subsequently be added to $\mathcal{O}$;
    \item otherwise orient the edge out from $v$, i.e., add $(w_0,v)$ to $\mathcal{O}$.
 \end{enumerate}
\smallskip

\noindent \fbox{\sc output} 
\medskip

The intervals $\{J(v):v\in V\}$ determine a semiorder.  Identifying $v$
with~$J(v)$ gives a semiorder, $P$, on $V$.  Define $\phi(c)=P$ and
$\eta(c)=\mathcal{O}$.
\hfill$\Box$

\begin{example}\label{example:superstables algorithm} Let $G$ be the house graph
  of Figure~\ref{fig:house graph}, and fix the vertex ordering, $v_1,\dots,v_5$.
  Then $c=(-1,0,0,0,0)$ is a quasi-superstable divisor on $G$.
  Figure~\ref{fig:algorithm} illustrates the application of the superstables
  algorithm to $c$ and the given vertex ordering, producing a $G$-semiorientation,
  $\mathcal{O}$.  The dotted lines emanating from each vertex to the exterior of
  the house graph represent the edges to the sink of $K(G)$.

  We have $c_{\mathrm{max}}=(2,2,3,2,3)$ for $K(G)$.  The algorithm starts
  with the configuration $b=c_{\mathrm{max}}-c$ in the top left corner of
  Figure~\ref{fig:algorithm}.  The vertex~$v_1$ is unstable in $b$, and firing
  produces two oriented edges and two unstable vertices, $v_2$ and $v_5$.  Since
  $v_2$ comes first in the vertex ordering, it is fired next.  The subscripts
  keep track of vertex-firing precedence.  Note that when $v_5$ is fired, the
  edge $\{v_3,v_5\}$ is marked blank since, by that time, $v_3$ is unstable.
  Proceeding clockwise around the diagram, the algorithm terminates at the
  bottom left corner.  
  
  Let $\tilde{c}:=c+1_V$, a superstable configuration on $K(G)$.  When the
  algorithm terminates, the resulting configuration is
  $c_{\mathrm{max}}-\tilde{c}$, which is recurrent by Proposition~\ref{duality}.
  Adding $1_V$ reproduces the starting configuration~$c$.  However, adding~$1_V$
  is equivalent to firing the sink of $K(G)$, which explains (via Dhar's burning
  algorithm) why every vertex was guaranteed to fire.
\end{example}
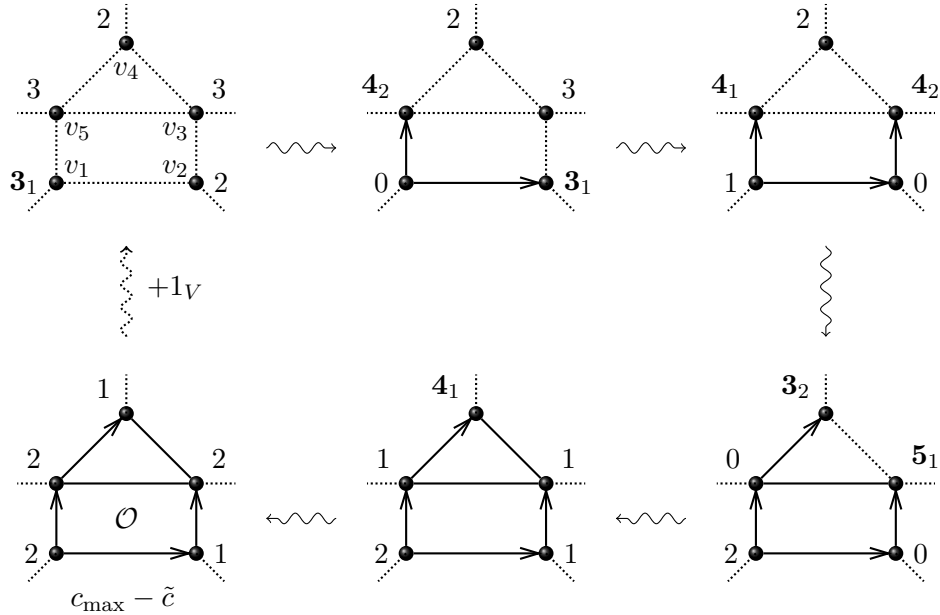
\begin{figure}[ht] 
\begin{tikzpicture}[scale=0.93]
\SetVertexMath
\GraphInit[vstyle=Art]
\SetUpVertex[MinSize=3pt]
\SetVertexLabel
\tikzset{VertexStyle/.style = {%
shape = circle,
shading = ball,
ball color = black,
inner sep = 2pt
}}
\SetUpEdge[color=black]

\begin{scope}[shift={(0,0)}]
\Vertex[LabelOut,Lpos=180,L=\mathbf{3}_1,x=0,y=0]{v_1}
\Vertex[LabelOut,Lpos=0,L=2,x=2,y=0]{v_2}
\Vertex[LabelOut,Lpos=45,L=3,x=2,y=1]{v_3}
\Vertex[LabelOut,Lpos=135,L=2,x=1,y=2]{v_4}
\Vertex[LabelOut,Lpos=135,L=3,x=0,y=1]{v_5}
\draw[densely dotted,thick] (0,0) -- (-0.4,-0.4);
\draw[densely dotted,thick] (2,0) -- (2.4,-0.4);
\draw[densely dotted,thick] (2,1) -- (2.56,1);
\draw[densely dotted,thick] (1,2) -- (1,2.56);
\draw[densely dotted,thick] (0,1) -- (-0.56,1);
\Edge[style={densely dotted,thick}](v_4)(v_5)
\Edge[style={densely dotted,thick}](v_5)(v_1)
\Edge[style={densely dotted,thick}](v_1)(v_2)
\Edge[style={densely dotted,thick}](v_2)(v_3)
\Edge[style={densely dotted,thick}](v_3)(v_4)
\Edge[style={densely dotted,thick}](v_5)(v_3)
\draw (0.3,0.2) node{$v_1$};
\draw (1.7,0.2) node{$v_2$};
\draw (1.7,0.74) node{$v_3$};
\draw (1,1.6) node{$v_4$};
\draw (0.3,0.74) node{$v_5$};
\end{scope}

\draw[decorate, decoration={snake, segment length=3mm, amplitude=2pt},->] (3,0.5)--(4,0.5);

\begin{scope}[shift={(5,0)}]
\Vertex[LabelOut,Lpos=180,L=0,x=0,y=0]{v_1}
\Vertex[LabelOut,Lpos=0,L=\mathbf{3}_1,x=2,y=0]{v_2}
\Vertex[LabelOut,Lpos=45,L=3,x=2,y=1]{v_3}
\Vertex[LabelOut,Lpos=135,L=2,x=1,y=2]{v_4}
\Vertex[LabelOut,Lpos=135,L=\mathbf{4}_2,x=0,y=1]{v_5}
\draw[densely dotted,thick] (0,0) -- (-0.4,-0.4);
\draw[densely dotted,thick] (2,0) -- (2.4,-0.4);
\draw[densely dotted,thick] (2,1) -- (2.56,1);
\draw[densely dotted,thick] (1,2) -- (1,2.56);
\draw[densely dotted,thick] (0,1) -- (-0.56,1);
\Edge[style={densely dotted,thick}](v_4)(v_5)
\Edge[style={->,>=mytip, thick}](v_1)(v_5)
\Edge[style={->,>=mytip, thick}](v_1)(v_2)
\Edge[style={densely dotted,thick}](v_2)(v_3)
\Edge[style={densely dotted,thick}](v_3)(v_4)
\Edge[style={densely dotted,thick}](v_5)(v_3)
\end{scope}

\draw[decorate, decoration={snake, segment length=3mm, amplitude=2pt},->] (8,0.5)--(9,0.5);

\begin{scope}[shift={(10,0)}]
\Vertex[LabelOut,Lpos=180,L=1,x=0,y=0]{v_1}
\Vertex[LabelOut,Lpos=0,L=0,x=2,y=0]{v_2}
\Vertex[LabelOut,Lpos=45,L=\mathbf{4}_2,x=2,y=1]{v_3}
\Vertex[LabelOut,Lpos=135,L=2,x=1,y=2]{v_4}
\Vertex[LabelOut,Lpos=135,L=\mathbf{4}_1,x=0,y=1]{v_5}
\draw[densely dotted,thick] (0,0) -- (-0.4,-0.4);
\draw[densely dotted,thick] (2,0) -- (2.4,-0.4);
\draw[densely dotted,thick] (2,1) -- (2.56,1);
\draw[densely dotted,thick] (1,2) -- (1,2.56);
\draw[densely dotted,thick] (0,1) -- (-0.56,1);
\Edge[style={densely dotted,thick}](v_4)(v_5)
\Edge[style={->,>=mytip, thick}](v_1)(v_5)
\Edge[style={->,>=mytip, thick}](v_1)(v_2)
\Edge[style={->,>=mytip, thick}](v_2)(v_3)
\Edge[style={densely dotted,thick}](v_3)(v_4)
\Edge[style={densely dotted,thick}](v_5)(v_3)
\end{scope}

\draw[decorate, decoration={snake, segment length=3mm, amplitude=2pt},->]
(11,-0.9)--(11,-2.2);

\begin{scope}[shift={(10,-5.3)}]
\Vertex[LabelOut,Lpos=180,L=2,x=0,y=0]{v_1}
\Vertex[LabelOut,Lpos=0,L=0,x=2,y=0]{v_2}
\Vertex[LabelOut,Lpos=45,L=\mathbf{5}_1,x=2,y=1]{v_3}
\Vertex[LabelOut,Lpos=135,L=\mathbf{3}_2,x=1,y=2]{v_4}
\Vertex[LabelOut,Lpos=135,L=0,x=0,y=1]{v_5}
\draw[densely dotted,thick] (0,0) -- (-0.4,-0.4);
\draw[densely dotted,thick] (2,0) -- (2.4,-0.4);
\draw[densely dotted,thick] (2,1) -- (2.56,1);
\draw[densely dotted,thick] (1,2) -- (1,2.56);
\draw[densely dotted,thick] (0,1) -- (-0.56,1);
\Edge[style={->,>=mytip,thick}](v_5)(v_4)
\Edge[style={->,>=mytip, thick}](v_1)(v_5)
\Edge[style={->,>=mytip, thick}](v_1)(v_2)
\Edge[style={->,>=mytip, thick}](v_2)(v_3)
\Edge[style={densely dotted,thick}](v_3)(v_4)
\Edge[style={thick}](v_5)(v_3)
\end{scope}

\draw[decorate, decoration={snake, segment length=3mm, amplitude=2pt},->]
(9,-4.8)--(8,-4.8);

\begin{scope}[shift={(5,-5.3)}]
\Vertex[LabelOut,Lpos=180,L=2,x=0,y=0]{v_1}
\Vertex[LabelOut,Lpos=0,L=1,x=2,y=0]{v_2}
\Vertex[LabelOut,Lpos=45,L=1,x=2,y=1]{v_3}
\Vertex[LabelOut,Lpos=135,L=\mathbf{4}_1,x=1,y=2]{v_4}
\Vertex[LabelOut,Lpos=135,L=1,x=0,y=1]{v_5}
\draw[densely dotted,thick] (0,0) -- (-0.4,-0.4);
\draw[densely dotted,thick] (2,0) -- (2.4,-0.4);
\draw[densely dotted,thick] (2,1) -- (2.56,1);
\draw[densely dotted,thick] (1,2) -- (1,2.56);
\draw[densely dotted,thick] (0,1) -- (-0.56,1);
\Edge[style={->,>=mytip,thick}](v_5)(v_4)
\Edge[style={->,>=mytip, thick}](v_1)(v_5)
\Edge[style={->,>=mytip, thick}](v_1)(v_2)
\Edge[style={->,>=mytip, thick}](v_2)(v_3)
\Edge[style={thick}](v_3)(v_4)
\Edge[style={thick}](v_5)(v_3)
\end{scope}

\draw[decorate, decoration={snake, segment length=3mm, amplitude=2pt},->]
(4,-4.8)--(3,-4.8);

\begin{scope}[shift={(0,-5.3)}]
\Vertex[LabelOut,Lpos=180,L=2,x=0,y=0]{v_1}
\Vertex[LabelOut,Lpos=0,L=1,x=2,y=0]{v_2}
\Vertex[LabelOut,Lpos=45,L=2,x=2,y=1]{v_3}
\Vertex[LabelOut,Lpos=135,L=1,x=1,y=2]{v_4}
\Vertex[LabelOut,Lpos=135,L=2,x=0,y=1]{v_5}
\draw[densely dotted,thick] (1,2) -- (1,2.56);
\draw[densely dotted,thick] (0,1) -- (-0.56,1);
\draw[densely dotted,thick] (0,0) -- (-0.4,-0.4);
\draw[densely dotted,thick] (2,0) -- (2.4,-0.4);
\draw[densely dotted,thick] (2,1) -- (2.56,1);
\Edge[style={->,>=mytip, thick}](v_5)(v_4)
\Edge[style={->,>=mytip, thick}](v_1)(v_5)
\Edge[style={->,>=mytip, thick}](v_1)(v_2)
\Edge[style={->,>=mytip, thick}](v_2)(v_3)
\Edge[style={thick}](v_3)(v_4)
\Edge[style={thick}](v_5)(v_3)
\draw (0.94,-0.65) node {$c_{\mathrm{max}}-\tilde{c}$};
\draw (1,0.5) node{$\mathcal{O}$};
\end{scope}

\draw[decorate, decoration={snake, segment length=3mm, amplitude=2pt},->,densely
dotted,thick]
(1,-2.2)--(1,-0.9);

\draw (1.7,-1.5) node {$+1_V$}; 
\end{tikzpicture}
\caption{Superstables algorithm (cf.~Example~\ref{example:superstables algorithm}). Start in the upper-left
corner.  Bold numbers connote unstable vertices.  Subscripts indicate
firing precedence.}\label{fig:algorithm}
\end{figure}

\begin{thm}\label{main}  
  Given $P\in\mathbb{I}$, for each vertex $v$ let $n_P(v)$ denote
    the cardinality of the set
    \[
    \{u\in V: \text{$u<v$ and $\{u,v\}\in E$}\}.
    \] 
    There are mappings
    \begin{align*}
      \theta:\mathbb{I}&\to\mathcal{S}\\
      P&\mapsto\sum_{v\in
      V}(n_P(v)-1)\,v,
    \end{align*}
   and
    \begin{align*}
      \psi:\mathbb{O}&\to\mathcal{S}\\
      \mathcal{O}&\mapsto\sum_{v\in
      V}(\mathrm{indeg_{\mathcal{O}}}(v)-1)\,v.
    \end{align*}
    Let $\pi:\mathcal{S}\times\mathfrak{S}\to\mathcal{S}$ be the first
    projection mapping, and let $\phi$ and $\eta$ be the mappings defined by the
    superstables algorithm.

    The following diagram commutes:
    \begin{equation}\label{diag1}
    \xymatrix{
    &\mathcal{S}\times\mathfrak{S}\ar@/_/[ddl]_\phi
    \ar@/^/[ddr]^\eta
    \ar@{>>}[d]^\pi&\\
    & \mathcal{S}&\\
    \mathbb{I}\ar@{>>}[ur]_(0.58)\theta\ar@{>>}[rr]_\nu &&\mathbb{O}.\ar@{>>}[ul]^(0.58)\psi
    }
  \end{equation}
    All mappings in the diagram are surjective except possibly for $\phi$ and $\eta$.
\end{thm}

\begin{proof}
\noindent{\bf  The image of $\phi$ is in $\mathbb{I}$.}  In stabilizing
$b$, each vertex eventually becomes unstable and is assigned an interval.  Thus,
the semiorder produced by the algorithm is a semiorder on all the vertices of
$G$.
\medskip

\noindent{\bf The image of $\theta$ is in $\mathcal{S}$.} Let $P\in\mathbb{I}$.
Then $P$ is isomorphic to the semiorder on a collection of intervals,
$\{J(v):v\in V\}$.  Letting $c=\theta(P)$, we must show that $c+1_V$ is a
superstable configuration of $K(G)$.   Let $c_{\mathrm{max}}$ be the maximal
stable configuration on $K(G)$, and define $b=c_{\mathrm{max}}-(c+1_V)$.  Note
that~$b\geq0$.   Starting with $b$ and firing the sink of $K(G)$ gives
$\tilde{b}=c_{\mathrm{max}}-c$.  By Propositions~\ref{duality}
and~\ref{dhar}, we must show that there is an ordering of $V$ forming a legal
firing sequence for $\tilde{b}$.

Choose any~$\sigma\in\mathfrak{S}$ so that $\sigma(i)<\sigma(j)$ if $\min
J_i<\min J_j$.  In particular, this means that if $v_i<v_j$ in $P$, then
$\sigma(i)<\sigma(j)$.  The $v$-th component of $\tilde{b}$ is
\[
\tilde{b}_v = \deg_{K(G)}(v) - n_P(v).
\]
We can legally fire the vertices in the order given by $\sigma$ since when it
becomes~$v$'s turn to fire, it will have received $n_P(v)$ grains of sand from the
firings of those vertices $u$ neighboring $v$ such that $u<v$ and will thus be
unstable. 
\medskip

\noindent{\bf We have $\theta\circ\phi=\pi$.}  Let $P=\phi(c,\sigma)$ and
$b=c_{\mathrm{max}}-c$ where $c_{\mathrm{max}}$ is the maximal stable
configuration on $K(G)$. Let $v\in V$.  If $v$ is unstable in~$b$, then $c_v=-1$
and there are no vertices smaller than $v$ in $P$. Hence, $n_P(v)-1=c_v=-1$, as
required. Otherwise, in the course of the algorithm, say $u$ is the vertex whose
firing causes $v$ to become unstable.  Then the vertices that are less than $v$
in $P$ are exactly the vertices $w$ such that $w\leq u$.  Among these $w$, only
those that are attached to $v$ by an edge contribute to making $v$ unstable.
Thus, exactly $n_P(v)$ grains of sand are added to $v$ in $b$ to make $v$
unstable. Since $b_v=\deg_{K(G)}v-1-c_v$, we have that $n_P(v)=1+c_v$, as
required.
\medskip

\noindent{\bf The mapping $\nu$ is surjective.}  The surjectivity of
$\nu$ follows from Theorem~\ref{region-semiorder}.
\medskip

\noindent{\bf We have $\psi\circ\nu=\theta$ and the image of $\psi$ is in
$\mathcal{S}$.}  Let $P\in\mathbb{I}$ and $\mathcal{O}=\nu(P)$.
Since~$\mathcal{O}$ is compatible with~$P$, we have that
$n_P(v)=\mathrm{indeg}_{\mathcal{O}}(v)$ for each $v\in V$.  Hence,
$\psi(\nu(P))=\theta(P)$ as mappings of configurations.  Since $\nu$ is
surjective, the image of $\psi$ is contained in the image of $\theta$, hence in~$\mathcal{S}$.
\medskip

\noindent{\bf We have $\nu\circ\phi=\eta$}.  Given
$(c,\sigma)\in\mathcal{S}\times\mathfrak{S}$, let $\mathcal{O}=\eta(c,\sigma)$
and~$P=\phi(c,\sigma)$.  We must show that $\mathcal{O}$ is compatible with $P$.
Let $b$ be as in the algorithm.  Run the algorithm up until it is a vertex $u$'s
turn to fire.  Suppose $\{u,v\}\in E$. Then~$v$ being stable at this point 
is equivalent to $u<v$ in $P$ and equivalent to~$(u,v)\in\mathcal{O}$.  
\end{proof}
\begin{cor}\label{cor-to-main}
  Let $\sigma\in\mathfrak{S}$.  Then the mappings
  \[
  \phi_{\sigma}\colon\mathcal{S}\to\mathbb{I},\qquad
  \eta_{\sigma}\colon\mathcal{S}\to\mathbb{O},\qquad
  \]
  defined by $\phi_{\sigma}(c)=\phi(c,\sigma)$ and
  $\eta_{\sigma}(c)=\eta(c,\sigma)$, are injective with left inverses~$\theta$
  and $\psi$, respectively.

  Let $\mathcal{S}_{\mathrm{max}}$ denote the maximal quasi-superstables under
  the relation ``$<$'' defined in Section~\ref{quasi-superstables}, and let
  $\mathbb{O}_{\mathrm{max}}$ denote the acyclic orientations
  of~$G$ (elements of $\mathbb{O}$ in which each edge is oriented).   For each
  $c\in\mathcal{S}$, define the {\em degree} of $c$ to be
  $\deg(c):=\sum_{v\in V}c_v$. Then,
  \begin{enumerate}
    \item\label{cor-to-main1} the restriction of $\eta_{\sigma}$ to $\mathcal{S}_{\mathrm{max}}$
      gives a bijection 
      \[ 
      \eta_{\sigma}\colon\mathcal{S}_{\mathrm{max}}\to\mathbb{O}_{\mathrm{max}}, 
      \]
    \item\label{cor-to-main2} for $c\in\mathcal{S}$, we have $\deg(c)\leq g-1$, where
      $g:=|E|-|V|+1$ is the {\em genus} of $G$,  with equality if and only if
      $c\in\mathcal{S}_{\mathrm{max}}$.
  \end{enumerate}
\end{cor}
\begin{proof}
  Both $\phi_{\sigma}$ and $\eta_{\sigma}$ are injective on $\mathcal{S}$ with
  the claimed left inverses since the restriction of $\pi$ in
  diagram~\ref{diag1} of Theorem~\ref{main} to $\mathcal{S}\times\{\sigma\}$ is
  bijective.  

  To show
  $\eta_{\sigma}(\mathcal{S}_{\mathrm{max}})\subseteq\mathbb{O}_{\mathrm{max}}$,
  let $c\in\mathcal{S}$ and suppose that
  $\mathcal{O}:=\eta_{\sigma}(c)\notin\mathbb{O}_{\mathrm{max}}$.  According to
  Theorem~\ref{region-semiorder} there is a point $t=(t_0,\dots,t_n)\in\R^{n+1}$
  and a corresponding $G$-semiorder, $P_t$, such that $\nu(P_t)=\mathcal{O}$.
  For ease of notation, we may assume that $t_0\leq\dots\leq t_n$.  Choose
  $t'\in\R^{n+1}$ such that $t'_i\geq t_i$ for all $i$ and $t'_{i+1}>t'_i+1$ for
  $0\leq i\leq n-1$.  Let $\mathcal{O}':=\nu(P_{t'})$ and
  $c':=\psi(\mathcal{O}')$.  Then the set of oriented edges $\mathcal{O}$ is a
  proper subset of $\mathcal{O}'$.  Hence, $c<c'$, showing
  $c\notin\mathcal{S}_{\mathrm{max}}$, as desired.

  Now let $c\in\mathcal{S}$, and choose any $\mathcal{O}\in\mathbb{O}$ such that
  $c=\psi(\mathcal{O})$.  From the definition of $\psi$, we have
  $\deg(c)=|\mathcal{O}|-|V|\leq |E|-|V|=g-1$, with equality exactly when
  ${\mathcal{O}\in\mathbb{O}_{\mathrm{max}}}$.  If $c<c'$ for some
  quasi-superstable $c'$, then since $\deg(c)<\deg(c')$, it follows that
  $\mathcal{O}\notin\mathbb{O}_{\mathrm{max}}$.  This shows, for instance, that
  $\psi(\mathbb{O}_{\mathrm{max}})\subseteq\mathcal{S}_{\mathrm{max}}$. 

  It remains to be shown that $\psi$ restricted to $\mathbb{O}_{\mathrm{max}}$
  is injective.  Let ${c\in\mathcal{S}_{\mathrm{max}}}$ and choose any
  $\mathcal{O}$ such that $c=\psi(\mathcal{O})$.  We have seen that
  $\mathcal{O}\in\mathbb{O}_{\mathrm{max}}$.  From the definition of $\psi$, we
  see that $\mathrm{indeg}_{\mathcal{O}}(v)$ is determined by~$c$ for each $v\in
  V$.  Then
  $\mathrm{outdeg}_{\mathcal{O}}(v)=\deg_G(v)-\mathrm{indeg}_{\mathcal{O}}$
  since $\mathcal{O}$ orients every edge of $G$, and as noted in~\cite{benson},
  these outdegrees determine $\mathcal{O}$.  To see this, consider the graph $G$
  oriented by $\mathcal{O}$. Since $\mathcal{O}$ is acyclic, this oriented graph
  must have sinks, i.e., vertices with outdegree $0$.  Thus, $c$ determines the
  orientation of all edges incident on these sinks.  Now remove these sinks and
  their incoming edges.  The orientation of the resulting graph is still
  acyclic, and its sinks may also be determined from $c$.  Iterate to see that
  $\mathcal{O}$ is determined by~$c$, and hence, $\psi$ is injective when
  restricted to $\mathbb{O}_{\mathrm{max}}$.
\end{proof}

\begin{thm}
The mapping $\nu$ is bijective if and only if $G$ is a complete graph.
\end{thm}
\begin{proof}
The surjectivity of $\nu$ is part of Theorem~\ref{main}.
Suppose that $G$ is a complete graph and that $\nu(P_1) = \nu(P_2)$ for some $P_1, P_2 \in
\mathbb{I}$.  Then $v_i < v_j$ in $P_1$ implies $(v_i,v_j) \in \nu(P_1)$, which in turn
implies $v_i < v_j$ in~$P_2$.  Similarly, $v_i < v_j$ in $P_2$ implies $v_i < v_j$ in~$P_1$.
Thus, $P_1 = P_2$, so $\nu$ is injective.

Now suppose $G$ is not complete, i.e., there exist $v_i,v_j \in V$ such that $\{v_i,
v_j\} \notin E$.  Let $P_1$ be the $G$-semiorder where no vertices are
comparable, and let $P_2$ be the $G$-semiorder where $v_i < v_j$ and all other pairs
of vertices are incomparable.  We have $\nu(P_1) = \emptyset =
\nu(P_2)$.  So $\nu$ is not injective.
\end{proof}

\begin{example}\label{example:not surjective}
The mappings $\phi$ and $\eta$ are not necessarily surjective.  Consider the
case where $G$ is a cycle graph on five vertices.  Figure~\ref{fig:surjectivity}
displays the Hasse diagram of a $G$-semiorder, $P$, and its compatible
$G$-semiorientation $\nu(P)=\mathcal{O}_P$.  For any ordering $\sigma$ of the
vertices, $\eta(\psi(\mathcal{O}_P),\sigma)=\mathcal{O}'$ where $\mathcal{O}'$
is as pictured in Figure~\ref{fig:surjectivity}.
\end{example}

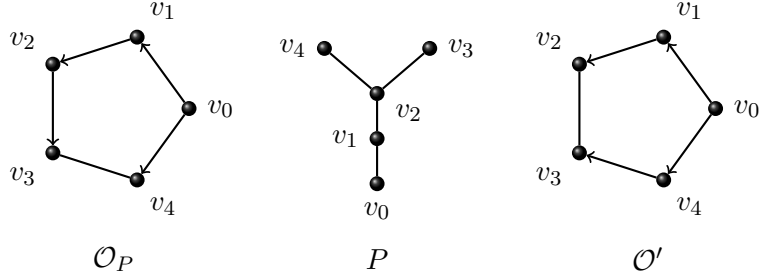
\begin{figure} 
\begin{tikzpicture}
\SetVertexMath
\GraphInit[vstyle=Art]
\tikzset{VertexStyle/.style = {%
shape = circle,
shading = ball,
ball color = black,
inner sep = 2pt
}}
\SetUpVertex[MinSize=3pt]
\SetVertexLabel
\SetUpEdge[color=black]
\Vertex[LabelOut,Lpos=0,
Ldist=.01cm,x=1,y=0]{v_0}
\Vertex[LabelOut,Lpos=72,
Ldist=.01cm,x=0.31,y=0.95]{v_1}
\Vertex[LabelOut,Lpos=144,
Ldist=.01cm,x=-0.81,y=0.59]{v_2}
\Vertex[LabelOut,Lpos=216,
Ldist=.01cm,x=-0.81,y=-0.59]{v_3}
\Vertex[LabelOut,Lpos=288,
Ldist=.01cm,x=0.31,y=-0.95]{v_4}
\Edge[style=->](v_0)(v_1)
\Edge[style=->](v_0)(v_4)
\Edge[style=->](v_1)(v_2)
\Edge[style=->](v_2)(v_3)
\Edge(v_3)(v_4)
\draw (0,-2) node{$\mathcal{O}_P$};

\Vertex[LabelOut,Lpos=270,
Ldist=.01cm,x=3.5,y=-1]{v_0}
\Vertex[LabelOut,Lpos=180,
Ldist=.01cm,x=3.5,y=-0.4]{v_1}
\Vertex[LabelOut,Lpos=356,
Ldist=.01cm,x=3.5,y=0.2]{v_2}
\Vertex[LabelOut,Lpos=0,
Ldist=.01cm,x=4.2,y=0.8]{v_3}
\Vertex[LabelOut,Lpos=180,
Ldist=.01cm,x=2.8,y=0.8]{v_4}
\Edges(v_0,v_1,v_2,v_3)
\Edges(v_2,v_4)
\draw (3.5,-2) node{$P$};

\Vertex[LabelOut,Lpos=0,
Ldist=.01cm,x=8,y=0]{v_0}
\Vertex[LabelOut,Lpos=72,
Ldist=.01cm,x=7.31,y=0.95]{v_1}
\Vertex[LabelOut,Lpos=144,
Ldist=.01cm,x=6.19,y=0.59]{v_2}
\Vertex[LabelOut,Lpos=216,
Ldist=.01cm,x=6.19,y=-0.59]{v_3}
\Vertex[LabelOut,Lpos=288,
Ldist=.01cm,x=7.31,y=-0.95]{v_4}
\Edge[style=->](v_0)(v_1)
\Edge[style=->](v_0)(v_4)
\Edge[style=->](v_1)(v_2)
\Edge[style=->](v_4)(v_3)
\Edge(v_2)(v_3)
\draw (7.1,-2) node{$\mathcal{O}'$};
\end{tikzpicture}
\caption{Figure for Example~\ref{example:not surjective}.}\label{fig:surjectivity}
\end{figure}

\subsection{Labeling regions with superstables}\label{labeling}
We can inductively label the regions of the $G$-semiorder arrangement,
$\mathscr{I}$, with quasi-superstables.  We call this labeling the (generalized)
{\em Pak-Stanley} labeling after \cite{pak-stanley}.  Start with the central region of
$\mathscr{I}$, the region defined by $|x_i-x_j|<1$ for all $i,j$ such that
$\{v_i,v_j\}$ is an edge of $G$.  Label this region with the divisor
$-1_V$.  Put the central region in a queue, $Q$.  Then, for as long as $Q$ is
not empty:
\begin{enumerate}
  \item Remove the first region $r$ from $Q$.
  \item For each unlabeled region $r'$ bordering~$r$:
    \begin{enumerate}
      \item Determine the unique indices $i\neq j$ such that $|x_i-x_j|<1$ in $r$ but $x_j > x_i + 1$ in
    $r'$.
  \item If $r$ is labeled by  $c=\sum_{k=0}^nc_kv_k$, then label $r'$ by
    $c'=c+v_j$. 
  \item Add $r'$ to the end of $Q$.  \hfill$\Box$
  \end{enumerate}
\end{enumerate}
For each $r\in\mathcal{R}$, define $\lambda(r)$ to be the label assigned to $r$
by the above algorithm.  Theorem~\ref{thm:labeling} guarantees that when
the algorithm terminates, the regions are labeled with quasi-superstables and
the labels are, in fact, independent of the order in which regions are removed
from the queue.

Recall the bijection $\rho\colon\mathbb{O}\to\mathcal{R}$, and let
$\tau=\rho^{-1}$.  Also recall the mapping~$\psi$ from Theorem~\ref{main} and
the mapping $\eta_{\sigma}$ from Corollary~\ref{cor-to-main}.
\begin{thm}\label{thm:labeling}
    We have $\lambda=\psi\circ\tau$. So there is a surjective mapping
\[
\xymatrix{
\lambda\colon\mathcal{R}\ar@{>>}[r]&\mathcal{S}
}
\]
and, for each $\sigma\in\mathfrak{S}$, a commutative diagram
\begin{equation}
\xymatrix{ 
\mathcal{S}\
\ar@{^{(}->}[rr]^{\eta_{\sigma}}&&\mathbb{O}\ar[d]^{\rho}_\shortparallel\\
  {}&&\mathcal{R}\ar@{>>}[llu]^{\lambda}.
  } 
\end{equation}
\end{thm}  
\begin{proof}
All the statements in the theorem follow directly from the equality
$\lambda=\psi\circ\tau$, which we now prove by induction.  
If $r$ is the central region then $\tau(r)$ is the partial orientation in
which all edges are marked blank.  Therefore, in this case, $\lambda(r)=\psi(\tau(r))$.

Let the labeling algorithm run, and suppose that the region $r$ has just been removed
from the queue.  By induction, suppose that $\lambda=\psi\circ\tau$ when
restricted to those regions that have been labeled so far.  Let $r'$ be an
unlabeled region bordering $r$.  Say that the inequalities that define $r$ are
the same as those that define $r'$ except that $|x_i-x_j|<1$ in $r$ and
$x_j>x_i+1$ in~$r'$.  It follows that if $\mathcal{O} = \tau(r)$ then $\tau(r') = \mathcal{O}
\cup \{(v_i,v_j)\}$.  Therefore, going from $\tau(r)$ to $\tau(r')$, only the
indegree of $v_j$ has increased by one, so
\begin{align*}
  \lambda(r')&=\lambda(r)+v_j&\qquad(\text{algorithm})\\
&=\psi(\tau(r))+v_j&\qquad(\text{induction})\\
&=\psi(\tau(r')).
\end{align*}
The result follows by induction.
\end{proof}

\section{Fixing a sink vertex}\label{sink}
In this section, we finally fully explain Figure~\ref{fig:example}.
For that figure, we started with a graph with a given sink vertex and
constructed a hyperplane arrangement with regions labeled by partial
orientations of the graph and sandpile configurations.  The nonnegative
configurations arising were exactly the superstables for the graph. 

Let $G$ be a graph with vertices $V=\{v_0,\dots,v_n\}$, and designate vertex $v_0$
as the sink.

\begin{definition}
The {\em $(G,v_0)$-semiorientations}, denoted $\mathbb{O}_0$, are the
$G$-semi\-orientations satisfying the additional requirement that $v_0$ is a
source:
\[
\mathbb{O}_0=\{\mathcal{O}\in\mathbb{O}: \mathrm{outdeg}_{\mathcal{O}}(v_0) =
\mathrm{deg}(v_0)\}.
\]
The set of {\em admissible} $(G,v_0)$-semiorientations is
\[
\widetilde{\mathbb{O}}_0=\{\mathcal{O}\in\mathbb{O}_0:\text{$\mathrm{indeg}_{\mathcal{O}}(v_i)
\geq 1$ for all $i \neq 0$}\}. 
\]
\end{definition}

Thus, while $v_0$ is the sink for the sandpile model on $G$---i.e., the sink for
the sake of defining the sandpile group, superstable configurations, and
$G$-parking functions---it is a source for any $\mathcal{O}\in\mathbb{O}_0$.

\begin{example}
  Figure~\ref{fig:example} shows the admissible $(G,q)$-semiorientations for the
  graph in Figure~\ref{fig:graph}. Due to lack of space, the requisite edges
  $(q,v_1)$ and $(q,v_2)$ are understood in Figure~\ref{fig:example}, but not
  drawn.
\end{example}

Our next goal is to describe the image of $\mathbb{O}_0$ under the mapping
$\psi\colon\mathbb{O}\to\mathcal{S}$ from Theorem~\ref{main}, thus accounting
for the configurations labeling the regions in Figure~\ref{fig:example}. We show
below that the image of $\mathbb{O}_0$ is the set of quasi-superstables 
assigning the value $-1$ to $v_0$ and whose only other negative values must occur at
vertices not connected to $v_0$ by an edge.  Further, the image of
$\widetilde{\mathbb{O}}_0
\subseteq A_0$ is the set of $G$-parking functions of $G$. 

Let
$\widetilde{V}=V\setminus\{v_0\}$, and let
\[
X=\{v\in\widetilde{V}:\{v,v_0\}\notin E\}.
\]
Define $1_X=\sum_{v\in X}v$, a configuration on $G$ (having chosen $v_0$ as the
sink).  Let $K(G)_0$ denote the graph $G$ but with an edge $\{v,v_0\}$ added for
each $v\in X$, and fix $v_0$ as its sink.  Thus, configurations on $G$ and on
$K(G)_0$ are elements of $\Z\widetilde{V}$, the free abelian group on
$\widetilde{V}$, a subgroup of~$\Z V$, the configurations on $K(G)$.  Recall
that $K(G)$ is the graph used to define quasi-superstables.
\begin{thm}\label{parking functions}\ 
  \begin{enumerate} 
    \item\label{ss/qss1} Define 
      \[
      \mathcal{S}_0=\{c\in\mathcal{S}:\text{$c_{v_0}=-1$ and $c+v_0\geq-1_X$}\}.
      \]
      Then
     \[ 
      \begin{array}[t]{rll}
      \psi(\mathbb{O}_0)&=&\{\tilde{c}-v_0: \text{$\tilde{c}+1_X$ a
      superstable on $K(G)_0$}\}\\[5pt]
      &=&\mathcal{S}_0.
     \end{array}
     \]
   \item\label{ss/qss2} Define
     \begin{align*}
       \widetilde{\mathcal{S}}_0&=\{c\in\mathcal{S}:\text{$c_{v_0}=-1$ and
         $c+v_0\geq0$}\}\subseteq\mathcal{S}_0.
     \end{align*}
     Then
     \begin{align*}
       \psi(\widetilde{\mathbb{O}}_0)&=\{\tilde{c}-v_0: \text{$\tilde{c}$ a superstable on $G$}\}\\
       &=\{\tilde{c}-v_0: \text{$\tilde{c}+1_X$ a superstable on $K(G)_0$ and
       $\tilde{c}\geq0$}\}\\[5pt]
       &=\widetilde{\mathcal{S}}_0.
     \end{align*}
     Thus, $\psi(\widetilde{\mathbb{O}}_0)=\widetilde{\mathcal{S}}_0$ is the set of $G$-parking
     functions with respect to~$v_0$.
  \end{enumerate}
\end{thm}
\begin{proof} We first prove 
  \medskip

  \begin{quote}
  {\bf Claim A}: If $c\in\Z\widetilde{V}$, then $c+1_X$ is superstable on $K(G)_0$ if
  and only if $c-v_0+1_V$ is superstable on $K(G)$ and $c\geq-1_X$.   
  \end{quote}
  \medskip

  For any graph $H$, let $E_H$ denote its edges. Recall that $\tilde{q}$ is the
  sink vertex for $K(G)$.  Let $c\in\Z\widetilde{V}$ with $c\geq-1_X$, and consider $c-v_0+1_V$
  as a configuration on $K(G)$.  Let $U\subseteq V$.  If $v_0\in U$, then we
  cannot legally fire $U$ from $c-v_0+1_V$: since $(c-v_0+1_V)_{v_0}=0$ and
  $\{v_0,\tilde{q}\}\in E_{K(G)}$, firing $U$ would result in a configuration
  with a negative $v_0$-component.  Thus, $c-v_0+1_V$ is superstable if and only
  if there are no nonnempty subsets $U\subseteq\widetilde{V}$ that can be
  legally fired.  

  For $v\in U\subseteq\widetilde{V}$, consider the edges incident with $v$ that lead
  out of $U$:
  \begin{align*}
    M(v,U)&=\left\{w\in(V\cup\{\tilde{q}\})\setminus U:\{v,w\}\in
    E_{K(G)}\right\}\\
    M(v,U)_0&=\left\{w\in V\setminus U:\{v,w\}\in E_{K(G)_0}\right\}.
  \end{align*}
  Then,
  \[
  M(v,U)\setminus M(v,U)_0=\{\tilde{q}\}
  \quad\text{and}\quad
  M(v,U)_0\setminus M(v,U)=
  \begin{cases}
    \emptyset&\text{if $v\notin X$},\\
    \{v_0\}&\text{if $v\in X$}.
  \end{cases}
  \]
  So the cardinality of $M(v,U)$ is
  \begin{equation}\label{outedges}
  |M(v,U)|=
  \begin{cases}
    |M(v,U)_0|+1&\text{if $v\notin X$},\\
    |M(v,U)_0|&\text{if $v\in X$}.
  \end{cases}
\end{equation}
Now $c-v_0+1_V$ is superstable on $K(G)$ if and only if $(c-v_0+1_V)_v<|M(v,U)|$
for all $v\in U$ for all nonempty $U\subseteq\widetilde{V}$, and $c+1_X$ is superstable on
$K(G)_0$ if and only if $(c+1_X)_v<|M(v,U)_0|$ for all $v\in U$ for all nonempty
$U\subseteq\widetilde{V}$.  For $v\in\widetilde{V}$ we have
\[
(c-v_0+1_V)_v=c_v+1=
\begin{cases} 
  (c+1_X)_v+1&\text{if $v\notin X$},\\
  (c+1_X)_v&\text{if $v\in X$}.
\end{cases}
\]
So Claim~A follows from~(\ref{outedges}).  The condition $c\geq-1_X$ is
required in the statement of the claim since superstables must be nonnegative.

Now let 
\[
  P=\{\tilde{c}-v_0: \text{$\tilde{c}+1_X$ superstable on $K(G)_0$}\}.
\]
The fact that $P=\mathcal{S}_0$ follows directly from Claim~A.  We now show that
$\psi(\mathbb{O}_0)=\mathcal{S}_0$ to finish the proof of part~(\ref{ss/qss1}).  Let
$\mathcal{O}\in\mathbb{O}_0$.  Then $\psi(\mathcal{O})\in\mathcal{S}$ by
Theorem~\ref{main}.  Since $v_0$ is a source for $\mathcal{O}$, we have
$\psi(\mathcal{O})_{v_0}=-1$, and if $v\in\widetilde{V}\setminus X$, then
$\psi(\mathcal{O})_v\geq0$.  Thus, $\psi(\mathcal{O})\in\mathcal{S}_0$.  Conversely,
given~$c\in\mathcal{S}_0$, run the superstables algorithm from Section~\ref{ss-alg} with
any vertex ordering of~$V$ in which $v_0$ appears first.   Using the notation
from the initialization stage of the algorithm, let $b=c_{\mathrm{max}}-c$.
Since $c_{v_0}=-1$, the vertex $v_0$ is unstable in~$b$ and will fire first.
Since $c+v_0\geq-1_X$, no vertex $v\in\widetilde{V}\setminus X$ is unstable in~$b$.
So when $v_0$ fires, the algorithm will orient each edge incident on $v_0$ out
from $v_0$.  If $\mathcal{O}$ is the semiorientation produced by the algorithm, it
follows that $\mathcal{O}\in\mathbb{O}_0$ and, by Theorem~\ref{main}, we have
$\psi(\mathcal{O})=c$.  Thus, $\psi(\mathbb{O}_0)=\mathcal{S}_0$.

To prove part~(\ref{ss/qss2}), let
\begin{align*}
  N&=\{\tilde{c}-v_0: \text{$\tilde{c}$ a superstable on $G$}\},\\
  \widetilde{P}&= \{\tilde{c}-v_0: \text{$\tilde{c}+1_X$ a superstable on $K(G)_0$ and
       $\tilde{c}\geq0$}\}\subseteq P.
\end{align*}

From part~(\ref{ss/qss1}), it follows directly that
$\psi(\widetilde{\mathbb{O}}_0)=\widetilde{P}=\widetilde{\mathcal{S}}_0$.  To
show $N=\widetilde{\mathcal{S}}_0$ and finish, proceed exactly as in the proof
of Claim~A.  Let $\tilde{c}\in\Z\widetilde{V}$ with~$\tilde{c}\geq0$.  We must
show that $\tilde{c}$ is superstable on $G$ if and only if $\tilde{c}-v_0+1_V$
is superstable on~$K(G)$.  Given $v\in U\subset\widetilde{V}$, this time
consider the set
\[
M(v,U)_G=\{w\in V\setminus U:\{v,w\}\in E_G\},
\]
and note that $|M(v,U)|=|M(v,U)_G|+1$, with $M(v,U)$ defined as before, from
which the result follows.
\end{proof}

Part~\ref{maximal ss2} of the following corollary recaptures a well-known result
from sandpile theory.  For instance, it occurs as Lemma~5 in~\cite{biggs} as a
statement about recurrent configurations, equivalent to our statement in light
of Proposition~\ref{duality}.
\begin{cor}
 Order the superstable configurations on $G$ by the relation ``$<$'' defined in
 Section~\ref{quasi-superstables}.  Let $\mathcal{S}_{\mathrm{max}}$ denote the
 maximal quasi-superstables on $G$, and let $g:=|E|-|V|-1$ be the genus of $G$,
 as in Corollary~\ref{cor-to-main}.  Let~$\tilde{c}$ be a superstable
 configuration on $G$.  Then, 
 \begin{enumerate}
   \item\label{maximal ss1} $\tilde{c}$ is maximal if and only if
     $\tilde{c}-v_0\in\widetilde{\mathcal{S}}_0\cap\mathcal{S}_{\mathrm{max}}$,
   \item\label{maximal ss2} $\deg(\tilde{c}):=\sum_{v\in\widetilde{V}}\tilde{c}_v\leq g$ with equality if and only
     if $\tilde{c}$ is maximal.
 \end{enumerate}
\end{cor}
\begin{proof}
  By Theorem~\ref{parking functions}~(\ref{ss/qss2}),
  $\tilde{c}-v_0\in\widetilde{\mathcal{S}}_0$.  Hence, part~(\ref{maximal ss2})
  follows immediately from part~(\ref{maximal ss1}) by
  Corollary~\ref{cor-to-main}~(\ref{cor-to-main2}).  We now
  prove part~(\ref{maximal ss1}).
  
  \noindent($\Rightarrow$)  Suppose
  $\tilde{c}-v_0\notin\mathcal{S}_{\mathrm{max}}$.  By Theorem~\ref{parking
  functions}~(\ref{ss/qss2}), there exists
  $\mathcal{O}\in\widetilde{\mathbb{O}}_0$ such that
  $\psi(\mathcal{O})=\tilde{c}-v_0$.  By Corollary~\ref{cor-to-main}, we have
  $\mathcal{O}\notin\mathbb{O}_{\mathrm{max}}$, i.e., $\mathcal{O}$ is not an
  acyclic orientation of $G$, but by the beginning of the proof to
  Corollary~\ref{cor-to-main}, there exists
  $\mathcal{O}'\in\mathbb{O}_{\mathrm{max}}$ such that $\mathcal{O}$ is a proper
  subset of $\mathcal{O}'$.  It follows that
  $\psi(\mathcal{O})=\tilde{c}-v_0<\psi(\mathcal{O}')=:c'$, where $c'\in\mathcal{S}$.  Since
  $\mathcal{O}\subset\mathcal{O}'$ and $\mathcal{O}'$ is acyclic, we have
  $c'_{v_0}=-1$ and $c'+v_0\geq0$, i.e., $c'\in\widetilde{\mathcal{S}}_0$.
  Define $\tilde{c\,}'$ by $c'=\tilde{c\,}'-v_0$.  Then~$\tilde{c\,}'$ is a
  superstable configuration on $G$ by Theorem~\ref{parking functions}, and
  $\tilde{c}<\tilde{c\,}'$.  So~$\tilde{c}$ is not a maximal superstable.

  \noindent($\Leftarrow$)  Conversely, suppose
  $\tilde{c}-v_0\in\mathcal{S}_{\mathrm{max}}$.  Take any superstable
  configuration~$\tilde{c\,}'$ such that $\tilde{c}\leq\tilde{c\,}'$.  Then
  $\tilde{c\,}'-v_0\in\mathcal{S}$ by Theorem~\ref{parking
  functions}~(\ref{ss/qss2}), and $\tilde{c}-v_0\leq\tilde{c\,}'-v_0$.  By
  maximality of $\tilde{c}-v_0$, it follows that $\tilde{c}=\tilde{c\,}'$.
  Hence,~$\tilde{c}$ is a maximal superstable.
\end{proof}
\begin{definition}\label{sink semiorder arrangement}
The {\em $(G,v_0)$-semiorder arrangement}, denoted $\mathscr{I}_0$, is the set
of hyperplanes in $\R^n$ given by
\[
x_i-x_j=1,
\]
for all $i,j$ not equal to $0$ such that $\{v_i, v_j\}\in E$.    
\end{definition}

\begin{definition}
The {\em regions} of $\mathscr{I}_0$, denoted $\mathcal{R}_0$, are the 
connected components of $\R^n\setminus\mathscr{I}_0$.
\end{definition}

Define the subset
\[
T_0=\{(x_0,\dots,x_n)\in\R^{n+1}:\text{$x_i>x_0+1$ whenever 
$\{v_i,v_0\}\in E_G$}\}
\]
and let
\[
\mathcal{R}'_0=\{r\in\mathcal{R}:r\subseteq T_0\}.
\]
The elements of $\mathcal{R}'_0$ are exactly those regions with corresponding
semiorientations (under $\rho$) having $v_0$ as a source.  Hence, the bijection
$\rho\colon\mathbb{O}\to\mathcal{R}$ restricts to a bijection
$\mathbb{O}_0\to\mathcal{R}_0'$.  The projection mapping
$(x_0,\dots,x_n)\to(x_1,\dots,x_n)$, omitting the $0$-th coordinate, induces a
bijection
\[
\pi_0\colon\mathcal{R}_0'\to\mathcal{R}_0.
\]
Therefore, we have the following theorem.
\begin{thm}\label{thm:affine bijection} The mapping
\[
\rho_0:=\pi_0\circ\rho\colon\mathbb{O}_0\to\mathcal{R}_0
\]
is a bijection.
\end{thm}
\begin{example}
  The $(G,v_0)$-semiorder arrangement for graph $G$ of
  Figure~\ref{fig:graph} with $v_0=q$ is drawn in Figure~\ref{fig:example}.  Its
  regions, $\mathcal{R}_0$, are the projections of the regions 
  $\mathcal{R}'_0$ displayed in Figure~\ref{fig:choice}.
\end{example}
The central region of $\mathscr{I}_0$ is the region defined by $|x_i - x_j| < 1$
for all distinct $i,j$ not equal to $0$ such that $\{v_i,v_j\}$ is an edge of
$G$.  Inductively label the regions of $\mathscr{I}_0$ as
in Section~(\ref{labeling}), but starting with the central region labeled with the
configuration that assigns $0$ to all $v_i$ such that $\{v_i,v_0\} \in E$ and
$-1$ to all other vertices, including $v_0$.  For each $r \in \mathcal{R}_0$, define
$\lambda_0(r)$ to be the label assigned to $r$ in this fashion.

Define $\tau_0=\rho_0^{-1}$.  There is a version of Theorem~\ref{labeling} in this
context (proved similarly):
\begin{thm}\label{thm:affine labeling}
We have $\lambda_0=\psi\circ\tau_0$. So there is a surjective mapping
\[
\xymatrix{
\lambda_0\colon\mathcal{R}_0\ar@{>>}[r]&\mathcal{S}_0
}
\]
and, for each $\sigma\in\mathfrak{S}_0$, a commutative diagram
\begin{equation}
\xymatrix{ 
\mathcal{S}_0\
\ar@{^{(}->}[rr]^{\eta_{\sigma,0}}&&\mathbb{O}_0\ar[d]^{\rho_0}_\shortparallel\\
  {}&&\mathcal{R}_0\ar@{>>}[llu]^{\lambda_0},
  } 
\end{equation}
where $\eta_{\sigma,0}$ is the restriction of $\eta_{\sigma}$ to
$\mathcal{S}_0$.
\end{thm}

For $i=1,\dots,n$, if $\{v_i,v_0\}\in E_G$, let
\[
\R^n_{(i,0)}=\R^n,
\]
otherwise, if $\{v_i,v_0\}\notin E_G$, let
\[
\R^n_{(i,0)}= \{(x_1,\dots,x_n)\in\R^n: \text{ $x_i>x_j+1$ for some $j$ with $\{v_i,v_j\}\in E$}\}.
\]
Define $\widetilde{\R}^n=\bigcap_{i=1}^n\R^n_{(i,0)}$.
\begin{definition}
  The {\em admissible regions} of $\mathscr{I}_0$, denoted
  $\widetilde{\mathcal{R}}_0$, are the connected components of
  $\widetilde{\R}^n\setminus\mathscr{I}_0$.
\end{definition}
Note that $\widetilde{\mathcal{R}}_0=\mathcal{R}_0$ exactly when each nonsink vertex is connected by an edge to the sink.

The admissible regions of $\mathscr{I}_0$ are exactly those regions whose
corresponding $G$-semiorientations satisfy
$\mathrm{indeg}_{\mathcal{O}}(v_i) \geq 1$ for all $i \neq 0$.  So combining
Theorem~\ref{parking functions} and Theorem~\ref{thm:affine labeling} gives
\begin{thm}\label{thm:Pak-Stanley}
\[
\lambda_0(\widetilde{\mathcal{R}}_0) = \widetilde{\mathcal{S}}_0
=\{c-v_0: \text{$c$ a superstable on $G$}\}.
\]
\end{thm}
\begin{example}
  The $9$ admissible regions in Figure~\ref{fig:example} are labeled by the $8$
  distinct superstables (or $G$-parking functions if one remembers that the sink
  is labeled by~$-1$).  The zero-configuration appears twice.
\end{example}

\section{Conclusion.}
Let $A$ be an $(n+1)\times(n+1)$ matrix.  Define $\mathcal{H}_A$ to be the set of
hyperplanes \[ x_i - x_j = A_{ij}, \] for all $i \neq j$ such that $\{v_i, v_j\}
\in E$.  For example, $\mathcal{H}_A = \mathscr{I}$ if $A$ has all $1$s as its
entries.

Define the \emph{regions} of $\mathcal{H}_A$, denoted $\mathcal{R}_A$, to be the
connected components of $\R^{n+1} \setminus \mathcal{H}_A$. The set of
inequalities $x_i-x_j<A_{ij}$ for all $i$ and $j$ defines the {\em central region} of $\mathcal{H}_A$.  We say that {\em
$\mathcal{H}_A$ has a central region} if this central region is nonempty. 

\begin{conj} Suppose that $\mathcal{H}_A$ has a central region.  Labeling the
  regions of $\mathcal{H}_A$ as in Section~\ref{labeling} defines a surjection
  \[ 
  \xymatrix{ 
    \lambda_A\colon\mathcal{R}_A\ar@{>>}[r]&\mathcal{S}.
    } 
  \]
\end{conj}
A similar conjecture holds if one first chooses $v_0$ as a sink: replace $A$
with an $n\times n$ matrix, and label regions as in Section~\ref{sink}.  So the
central region would be labeled with the configuration that assigns $0$ to
vertices connected to $v_0$ and $-1$ to the other vertices (including $v_0$).
We conjecture that the nonnegative configurations that arise as labels are
exactly the $G$-parking functions.  The $G$-Shi conjecture of Duval, Klivans,
and Martin is a special case.

In the spirit of \cite{athanasiadis} and \cite{pak-stanley} , it would be
interesting to extend our results to the case of multigraphs: graphs in which
multiple edges are allowed between vertices.  If there are $k$ edges between
$v_i$ and $v_j$, one might replace the two hyperplanes $x_i-x_j=\pm1$ with the
$2k$ hyperplanes $x_i-x_j=\pm1,\dots,\pm k$.
\bibliography{soap}{}
\bibliographystyle{plain}
\end{document}